\documentclass[11pt,a4paper,fleqn]{article}

\usepackage{amsmath,amssymb,amsthm,enumerate,cite,tikz}
\usepackage{longtable}

\newcommand{\ev}{\mathop{\mathbf{ev}}\nolimits}
\newcommand{\F}{\mathbb{F}}

\newcommand{\GlnF}{\mathop{{\rm GL}(n,\F)}\nolimits}
\newcommand{\GlF}{\mathop{{\rm GL}(3,\F)}\nolimits}
\newcommand{\GlC}{\mathop{{\rm GL}(3,\mathbb{C})}\nolimits}

\newcounter{tbn}
\newcounter{pict}

\newcounter{mcasenum}

\newtheorem{theorem}{Theorem}[section]
\newtheorem{lemma}[theorem]{Lemma}
\newtheorem*{lemma*}{Lemma}

\newtheorem*{corollary*}{Corollary}
\newtheorem{proposition}[theorem]{Proposition}
{\theoremstyle{definition} \newtheorem{definition}[theorem]{Definition}
\newtheorem*{definition*}{Definition}
\newtheorem{example}[theorem]{Example}
\newtheorem*{example*}{Example}
\newtheorem*{examples*}{Exampes}
\newtheorem{remark}[theorem]{Remark}
\newtheorem*{remark*}{Remark}

\newtheorem*{result*}{Result}
\newtheorem*{comment*}{Comment}

\begin{document}

\title{\bf Degenerations of complex associative algebras of dimension three via Lie and Jordan algebras
}


\author{N. M. Ivanova%
\thanks{
\textit{\noindent Institute of Mathematics of NAS of Ukraine,~3 Tereshchenkivska Str., 01601 Kyiv, Ukraine and 
European University of Cyprus, Nicosia, Cyprus.}
{E-mail: ivanova.nataliya@gmail.com}
}
\ and\ \
C. A. Pallikaros%
\thanks{%
\textit{Department of Mathematics and Statistics,
University of Cyprus,
P.O.Box 20537,
1678 Nicosia,
Cyprus.}
{E-mail: pallikar@ucy.ac.cy}
}
}
\date{December 20, 2022}

\maketitle

\begin{abstract}
Let $\boldsymbol\Lambda_3(\mathbb C)\,(=\mathbb C^{27})$ be the space of structure vectors of $3$-dimensional algebras over $\mathbb C$ considered as a $G$-module via the action of $G=\GlC$ on $\boldsymbol\Lambda_3(\mathbb C)$ `by change of basis'.
We determine the complete degeneration picture inside the algebraic subset $\mathcal A^s_3$ of $\boldsymbol\Lambda_3(\mathbb C)$ consisting of associative algebra structures via the corresponding information on the algebraic subsets $\mathcal L_3$ and $\mathcal J_3$ of $\boldsymbol\Lambda_3(\mathbb C)$ of Lie and Jordan algebra structures respectively.
This is achieved with the help of certain $G$-module endomorphisms $\phi_1$, $\phi_2$ of $\boldsymbol\Lambda_3(\mathbb C)$ which map $\mathcal A^s_3$ onto algebraic subsets of $\mathcal L_3$ and $\mathcal J_3$ respectively.
\end{abstract}

\section{Introduction} 

The notion of degeneration or contraction arises in various physical investigations.
It was first introduced by Segal~\cite{Segal1951} and In\"on\"u and Wigner~\cite{InonuWigner1953,InonuWigner1954} in the case of Lie groups and Lie algebras in order to link certain properties of the classical mechanics, the relativistic mechanics and the quantum mechanics.
The main idea involved was to try to obtain certain properties of one of the physical theories using the corresponding properties of another theory via a kind of limiting process.
It turned out that the classical mechanics can be studied as a limit case of quantum mechanics  as the Planck constant tends to zero.
The symmetry group of relativistic mechanics (the Poincar\'e group) can be viewed as a contraction of the symmetry group of classical mechanics (the Galilean group) if we assume that the light velocity $c\to\infty$.
The notion of degeneration also has applications in other branches of mathematics. 

The present paper is concerned with the investigation of degenerations within certain classes of algebras.
The first classification of complex low-dimensional associative algebras has been made by B. Peirce in 1870, see also~\cite{Peirce1881}.
Although this classification contains only the so called ``pure algebras'', the approach used by Peirce can be generalized to more general classes of algebras.
Complex associative 3-dimensional algebras with a unit were classified by P. Gabriel in~\cite{Gabriel1975}.
Moreover, in~\cite{Gabriel1975} Gabriel also constructs all degenerations within the class of 2-dimensional complex associative algebras and all degenerations within the classes of 3- and 4-dimensional complex associative algebras with a unit.
Modern classification of all complex 3-dimensional associative algebras can be found in~\cite{FialowskiPenkava2009,KobayashiShirayanagiTsukada2021}.

Jordan algebras of dimension three over an algebraically closed field of characteristic not equal to 2 or 3 are classified in~\cite{KashubaShestakov2007}, where the authors also determine the irreducible components of the variety of 3-dimensional Jordan algebras.
See also~\cite{GorshkovKaygorodovPopov2021} for a description of the degenerations within the variety of complex 3-dimensional Jordan algebras.
The degenerations within the variety of complex 3-dimensional Lie algebras have been constructed in~\cite{Agaoka1999,NesterenkoPopovych2006}.

It will be convenient at this point to introduce some notation.
Let $\boldsymbol\Lambda_3(\mathbb C)\,(=\mathbb C^{27})$ be the space of structure vectors of complex 3-dimensional algebras.
We can consider $\boldsymbol\Lambda_3(\mathbb C)$ as a $\mathbb C G$-module via the natural (linear) action of $G=\GlC$ on $\boldsymbol\Lambda_3(\mathbb C)$ `by change of basis'.
We say that there is a degeneration from $\boldsymbol\lambda$ to $\boldsymbol\mu$ (with $\boldsymbol\lambda,\boldsymbol\mu\in\boldsymbol\Lambda_3(\mathbb C)$) if $\boldsymbol\mu$ belongs to the Zariski-closure of the $G$-orbit of $\boldsymbol\lambda$.

The aim of the present paper is to determine the complete degeneration picture  inside the algebraic subset $\mathcal A^s_3$ of $\boldsymbol\Lambda_3(\mathbb C)$ consisting of the associative algebra structures, via the corresponding information on the algebraic subsets $\mathcal L_3$ and $\mathcal J_3$ of $\boldsymbol\Lambda_3(\mathbb C)$ of Lie and Jordan algebras respectively.
In order to achieve this, we define certain $\mathbb C G$-module endomorphisms $\phi_1$, $\phi_2$ of $\boldsymbol\Lambda_3(\mathbb C)$ which map $\mathcal A_3^s$ onto algebraic subsets of $\mathcal L_3$ and $\mathcal J_3$ respectively.
A key role in our approach is played by the explicit computation of the $B$-orbit (where $B$ is a Borel subgroup of $G$) of appropriate elements of $\boldsymbol\Lambda_3(\mathbb C)$ and the consideration of the intersection of the closure of these $B$-orbits with certain algebraic subsets of $\boldsymbol\Lambda_3(\mathbb C)$ which played some part in~\cite{IvanovaPallikaros2019} and~\cite{PallikarosWard2020}.
Locating various polynomials in the vanishing ideal of such orbits is one of the important ingredients in some of our arguments.
This approach not only allowed us to rule out the possibility of degeneration in certain cases where this was not easy to achieve via the various necessary conditions for degeneration  commonly used in literature, but also provided a very practical means of constructing degenerations in certain cases where a degeneration actually exists.
The idea of explicitly computing an orbit and locating polynomials in the vanishing ideal was already used in~\cite{IvanovaPallikaros2019b} and this led, as a by-product, to the determination of various degenerations between 3-dimensional Lie algebras over an arbitrary field.

The paper is organized as follows:
In Section~\ref{SectionPrelim} we include some preliminary lemmas and discuss some of their applications via which we give the flavour of the general techniques that will be used later on in the paper.
In Sections~\ref{SectionNecCond} and~\ref{SectionPhi1Phi2} we recall some necessary conditions for degeneration and also the defining conditions for certain algebraic sets which will play a key role in the paper.
Moreover, in Section~\ref{SectionPhi1Phi2} we discuss some basic properties of the maps $\phi_1$ and $\phi_2$.
In Section~\ref{Section3dimAlg} we recall certain results regarding the varieties $\mathcal A^s_3$, $\mathcal L_3$ and $\mathcal J_3$ in the framework that has been built in the earlier sections.
Finally, in Section~\ref{SectionDegen3dim} which contains the main results of the paper, the degeneration picture inside the variety $\mathcal A^s_3$ is completely determined.

We remark that in~\cite{MohammedRakhimovShSaidHusain2017}, the authors aim to determine the irreducible components of the variety~$\mathcal A^s_3$.
Their approach is based entirely on the computation of various algebra invariants, which they use together with certain necessary conditions for degeneration, in order to rule out the possibility of degeneration between various algebras.
For the purposes of their work they do not include any constructions of degenerations or comment further in the cases where the possibility of degeneration is open.
There are, however, some inaccuracies in their computations of certain algebra invariants, sometimes leading to inaccuracies regarding the possibility of degeneration.

\section{Preliminaries}\label{SectionPrelim}
Throughout this paper $\F$ denotes an arbitrary infinite field and $n$ a positive integer.
We also let  $G=\GlnF$. 
We fix $V$ to be a finite dimensional $\F$-vector space with $\dim_{\F}V=n$.
We call $\mathfrak{g}$ an algebra structure on~$V$ if $\mathfrak{g}$ is an $\F$-algebra having $V$ as its underlying vector space
(so $\mathfrak{g}$ is a not necessarily associative algebra which has multiplication defined via a suitable $\F$-bilinear map $[,]_{\mathfrak{g}}\colon V\times V\to V:$ $(u,v)\mapsto [u,v]_{\mathfrak{g}}$ for $u,v\in V$).
We denote by~$\boldsymbol A$ the set of all algebra structures on~$V$.

If $(u_1,\ldots,u_n)$ is an ordered $\F$-basis of $V$, the multiplication in $\mathfrak g=(V, [,])\in\boldsymbol A$ is completely determined by the structure constants $\alpha_{ijk}\in\F$ ($1\le i,j,k\le n$) given by $[u_i,u_j]=\sum_{k=1}^n\alpha_{ijk}u_k$.
We will regard this set of structure constants $\alpha_{ijk}$ as an ordered $n^3$-tuple by imposing an ordering on the ordered triples $(i,j,k)$, for example we could choose the lexicographic ordering.
We call the ordered $n^3$-tuple $\boldsymbol\alpha=(\alpha_{ijk})\in\F^{n^3}$ the structure vector of $\mathfrak g\in\boldsymbol A$ relative to the $\F$-basis $(u_1,\ldots,u_n)$ of~$V$.
Also denote by $\boldsymbol\Lambda\,(=\F^{n^3})$ the set $\{\boldsymbol\lambda=(\lambda_{ijk})\in\F^{n^3}\colon \boldsymbol\lambda$ occurs as the structure vector of some $\mathfrak g\in\boldsymbol A$ relative to some ordered basis of $V\}$.
Clearly the structure vector $\boldsymbol\lambda\in\boldsymbol\Lambda$ occurs as the structure vector of both $\mathfrak g_1,\mathfrak g_2\in\boldsymbol A$ (relative to suitable $\F$-bases of $V$) if, and only if,  the algebra structures $\mathfrak g_1$ and $\mathfrak g_2$ are $\F$-isomorphic.
In what follows we will write $\mathfrak g_1\simeq\mathfrak g_2$ to denote that algebras $\mathfrak g_1,\mathfrak g_2\in\boldsymbol A$ are $\F$-isomorphic.

The set $\boldsymbol\Lambda\,(=\F^{n^3})$ forms an $\F$-vector space via the usual (componentwise) addition and scalar multiplication.
We will use symbol $\bf abc$ to denote the member $\boldsymbol{\lambda}\,(=(\lambda_{ijk}))$ of $\boldsymbol\Lambda$ having $\lambda _{abc}=1$ and all other $\lambda _{ijk}$ equal to $0$.
We will refer to the $\F$-basis of $\boldsymbol{\Lambda}$ consisting of the $n^3$ structure vectors of this form as the standard basis of $\boldsymbol{\Lambda}$.

We can also regard $\boldsymbol A$ as an $\F$ vector space:
For $\mathfrak{g}_1,\mathfrak{g}_2\in\boldsymbol A$ with $\mathfrak{g}_1=(V,[,]_1)$, $\mathfrak{g}_2=(V,[,]_2)$ and for $\alpha\in\F$ define $\mathfrak{g}_1+\mathfrak{g}_2=(V,[,])\in\boldsymbol A$ where
$[u,v]=[u,v]_1+[u,v]_2$ and $\alpha \mathfrak{g}_1=(V,[,]_\alpha)$ where $[u,v]_\alpha=\alpha[u,v]_1$ for all $u,v\in V$.

\medskip
For the rest of the paper it will also be convenient to fix an ordered $\F$-basis $(e_1,\ldots,e_n)$ of $V$ which we will call the standard basis of $V$.

\medskip
We can then obtain an isomorphism of $\F$-vector spaces $\Theta:\boldsymbol A\to\boldsymbol\Lambda$ where, for $\mathfrak{g}\in\boldsymbol A$ we define $\Theta(\mathfrak g)\,(\in\boldsymbol\Lambda)$ to be the structure vector of $\mathfrak g$ relative to the standard basis $(e_1,\ldots,e_n)$ of $V$.

With the help of the isomorphism $\Theta:\boldsymbol A\to\boldsymbol\Lambda$ we define the map $\Omega:\boldsymbol\Lambda\times G\to\boldsymbol\Lambda:$ $(\boldsymbol\lambda,g)\mapsto\boldsymbol\lambda g$ ($\boldsymbol\lambda\in\boldsymbol\Lambda$, $g=(g_{ij})\in G$) where
$\boldsymbol\lambda g\in\boldsymbol\Lambda$ is the structure vector of $\Theta^{{}^{-1}}(\boldsymbol\lambda)\in\boldsymbol A$
relative to the $\F$-basis $(v_1,\ldots,v_n)$ of $V$ given by $v_j=\sum_{i=1}^ng_{ij}e_i$, for $1\le j\le n$.
(In particular, $g\in G$ is the transition matrix from the basis $(e_i)_{i=1}^n$ to the basis $(v_i)_{i=1}^n$ of $V$.)

It is easy to observe that the map $\Omega$ defines a linear right action of $G$ on $\boldsymbol\Lambda$ and that the resulting orbits of this action correspond precisely to the isomorphism classes of $n$-dimensional $\F$-algebras.
Also note that the map $\Omega$ gives $\boldsymbol\Lambda$ the structure of a right $\F G$-module.
The orbit of $\boldsymbol\lambda\in\boldsymbol\Lambda$ with respect to the above $G$-action will be denoted by $O(\boldsymbol\lambda)$, (so $O(\boldsymbol\lambda)=\boldsymbol\lambda G$).

\medskip
Next, we recall briefly some basic facts on algebraic sets.

Let $\F[{\bf X}]$ be the ring $\F[X_{ijk}:\ 1\le i,j,k\le n]$ of polynomials in the indeterminates $X_{ijk}$ ($1\le i,j,k\le n$) over $\F$.
For each $\boldsymbol\lambda=(\lambda_{ijk})\in\boldsymbol\Lambda$ we can define the evaluation map $\ev_{\boldsymbol\lambda}:\F[{\bf X}]\to \F$
to be the unique ring homomorphism $\F[{\bf X}]\to \F$ such that $X_{ijk}\mapsto \lambda_{ijk}$ for $1\le i,j,k\le n$
and which is the identity on $\F$.
A subset $W$ of $\boldsymbol\Lambda$ is algebraic (and thus closed in the Zariski topology on $\boldsymbol\Lambda$) if there exists a subset
$S\subseteq\F[{\bf X}]$ such that
$W=\{\boldsymbol\lambda=(\lambda_{ijk})\in\boldsymbol\Lambda:$ $\ev_{\boldsymbol\lambda}(f)=0$ for all $f\in S\}$.
The Zariski closure of a subset $Y$ of $\boldsymbol\Lambda$ will be denoted by $\overline{Y}$.
By a closed subset of $\boldsymbol\Lambda$ we will always mean a Zariski-closed subset of $\boldsymbol\Lambda$.
Finally, for $U\subseteq\boldsymbol\Lambda$, the vanishing ideal ${\bf I}(U)$ of $U$ is defined by
${\bf I}(U)=\{f\in\F[{\bf X}]:\ \ev_{\boldsymbol\lambda}(f)=0$ for all $\boldsymbol\lambda\in{U}\}$.

\begin{definition}\label{DefDegeneration}
Let $\mathfrak{g}_1, \mathfrak{g}_2\in\boldsymbol A$.
We say that $\mathfrak{g}$ degenerates to $\mathfrak{h}$ (respectively, $\mathfrak{g}$ properly degenerates to $\mathfrak{h}$) if $\Theta(\mathfrak h) \in\overline{O(\Theta(\mathfrak g) )}$
(respectively, $\Theta(\mathfrak h) \in\overline{O(\Theta(\mathfrak g) )}- O(\Theta(\mathfrak g) )$).
[Note that, for $\boldsymbol\lambda,\boldsymbol\mu\in\boldsymbol\Lambda$, we have that $O(\boldsymbol\mu)\subseteq\overline{O(\boldsymbol\lambda)}$ whenever $\mu\in\overline{O(\boldsymbol\lambda)}$, see~\cite{IvanovaPallikaros2019}.]
We  write $\boldsymbol\lambda\to\boldsymbol\mu$ (and $\Theta^{-1}(\boldsymbol\lambda)\to\Theta^{-1}(\boldsymbol\mu)$) if for $\boldsymbol\lambda,\boldsymbol\mu\in\boldsymbol\Lambda$ we have that $\boldsymbol\mu\in\overline{O(\boldsymbol\lambda)}$.
Similarly, we write $\boldsymbol\lambda\not\to\boldsymbol\mu$ (and $\Theta^{-1}(\boldsymbol\lambda)\not\to\Theta^{-1}(\boldsymbol\mu)$) if $\boldsymbol\mu\not\in\overline{O(\boldsymbol\lambda)}$.
\end{definition}
A well-known result is that  $\mathfrak g$ degenerates to the Abelian algebra (the zero algebra) for all $\mathfrak g\in\boldsymbol A$.
In this paper we will consider (and also compare) degenerations within certain classes of algebras.

\medskip
For the rest of this section we prove some preliminary lemmas and discuss certain of their applications which involve techniques that will be used in Section~\ref{SectionDegen3dim} where  the main results of the paper are proved.

\begin{lemma}\label{LemmaA}
Let $f\colon\F\to\boldsymbol\Lambda$ be a continuous function in the Zariski topology.
Also let $U$ be a finite subset of $\F$ and let $S=\F-U$.
Then $f(u)\in \overline{f(S)}$ for all $u\in U$.
\end{lemma}
\begin{proof}
The hypothesis that $f$ is continuous ensures that $f(\bar S)\subseteq \overline{f(S)}$.
Hence, it suffices to show that $\bar S=\F$.
Suppose, on the contrary, that $\bar S\subsetneqq\F$.
Then $\bar S=\F-U'$ where $\varnothing\subsetneqq U'\subseteq U$ which, in turn, gives $\F=\bar S\cup U'$.
This is a contradiction as $\bar S$ and $U'$ are both closed subsets of $\F$ with $\varnothing\subsetneqq \bar S\subsetneqq \F$ and $\varnothing\subsetneqq U'\subsetneqq \F$ and $\F$ is irreducible (see, for example,~\cite[Example~1.1.13]{Geck2003}).
We conclude that $\bar S=\F$.
\end{proof}

\begin{example}\label{ExamplApplA}
(An application of Lemma~\ref{LemmaA}.)
Let $n=3$ and suppose ${\rm char}\,\F\ne2$.
Let $\boldsymbol\lambda={\bf221}+{\bf331}+2({\bf321})\in\boldsymbol\Lambda$ and, for each $t\in\F-\{0\}$, let
$g(t)=${\footnotesize$
\begin{pmatrix}
-t&0&0\\
0&0&-1\\
0&t&1
\end{pmatrix}$}$\in{\mathop{{\rm GL}(3,\F)}\nolimits}$.
Then $\boldsymbol\lambda g(t)={\bf231}-{\bf321}-t({\bf221})$ with $t\ne0$.
Moreover, the map $f\colon\F\to\boldsymbol\Lambda\,(=\F^{27}):$ $t\mapsto{\bf231}-{\bf321}-t({\bf221})$, ($t\in\F$) is continuous in the Zariski topology.
Set $S=\F-\{0\}$.
Clearly $f(S)\subseteq\boldsymbol\lambda G$.
Invoking Lemma~\ref{LemmaA} we get that ${\bf231}-{\bf321}\,(=f(0))\in\overline{f(S)}\subseteq\overline{\boldsymbol\lambda G}$.
\end{example}

\begin{lemma}\label{LemmaB}(Compare with~\cite[Proposition~1.7]{GrunewaldOHalloran1988a}).
Let $P$ be a parabolic subgroup of $G$ with $\F$ algebraically closed.
Also let $\boldsymbol\lambda \in\boldsymbol\Lambda$.
Then $(\overline{\boldsymbol\lambda P})G=\overline{\boldsymbol\lambda G}\,(=\overline{O(\boldsymbol\lambda)})$.
\end{lemma}

\begin{proof}
Assume the hypothesis.
Clearly $\boldsymbol\lambda P\subseteq \boldsymbol\lambda G\subseteq \overline{\boldsymbol\lambda G}$ so $\overline{\boldsymbol\lambda P}\subseteq \overline{\boldsymbol\lambda G}$.
It follows that $(\overline{\boldsymbol\lambda P})G\subseteq \overline{\boldsymbol\lambda G}$ since $\overline{\boldsymbol\lambda G}$ is a union of $G$-orbits by~\cite[Prorosition~2.5.2(a)]{Geck2003}.
Next, we let $C=\overline{\boldsymbol\lambda P}$.
Then $C$ is closed and it is also $P$-invariant, since $C$ ia a union of $P$-orbits again by~\cite[Prorosition~2.5.2(a)]{Geck2003}.
Clearly, $\boldsymbol\lambda G\subseteq(\overline{\boldsymbol\lambda P})G=CG$.
From~\cite[Corollary~3.2.12(a)]{Geck2003}, $CG\,(\subseteq\boldsymbol\Lambda)$ is a closed set.
Hence, $\overline{\boldsymbol\lambda G}\subseteq CG=(\overline{\boldsymbol\lambda P})G$. We conclude that $(\overline{\boldsymbol\lambda P})G=\overline{\boldsymbol\lambda G}$.
\end{proof}

\begin{remark}\label{RemarkBorelSubgroup}
Let $\F$ be algebraically closed.
Also let $B$ be a Borel subgroup of $G$.

(i) The conclusion of Lemma~\ref{LemmaB} still holds with $B$ in the place of $P$ since every Borel subgroup of $G$ is parabolic by Borel's theorem,
see for example,~\cite[Theorem~3.4.3]{Geck2003}.

(ii) Suppose $U$ is a subset of $\boldsymbol\Lambda$ which is also a union of $G$-orbits.
If $\boldsymbol\lambda\in\boldsymbol\Lambda$, we then have that $U\cap\overline{O(\boldsymbol\lambda)}=\varnothing$ whenever $U\cap\overline{\boldsymbol\lambda B}=\varnothing$.
To see this, suppose that $U\cap\overline{O(\boldsymbol\lambda)}\ne\varnothing$ and that $\boldsymbol\nu\in U\cap\overline{O(\boldsymbol\lambda)}$.
Since $\boldsymbol\nu\in \overline{O(\boldsymbol\lambda)}$, Lemma~\ref{LemmaB} ensures that $\boldsymbol\nu\in O(\boldsymbol\mu)$ for some $\boldsymbol\mu\in\overline{\boldsymbol\lambda B}$.
Hence $\boldsymbol\mu\in O(\boldsymbol\nu)$.
Since $U$ is a union of $G$-orbits and $\boldsymbol\nu\in U$, we have $O(\boldsymbol\nu)\subseteq U$ and hence $\boldsymbol\mu\in U$.
This leads to $U\cap \overline{\boldsymbol\lambda B}\ne\varnothing$, since $\boldsymbol\mu\in U\cap \overline{\boldsymbol\lambda B}$.
\end{remark}

\begin{example}\label{ExamplApplB}
(An application of Lemma~\ref{LemmaB}.)

Let $n=3$ and let $\F=\mathbb C$.
Also let $\beta\in\mathbb F-\{-1\}$ and set $\boldsymbol\lambda={\bf231}+\beta({\bf321})\in\boldsymbol\Lambda$.
We aim to show that $\overline{O(\boldsymbol\lambda)}\cap\mathcal K_3=\{{\bf0}\}$, where $\mathcal K_3\,(\subseteq\boldsymbol\Lambda)$ is given by $\Theta^{-1}(\mathcal K_3)=\{\mathfrak g=(V,[,])\in\boldsymbol A\colon [u,u]=0_V\ \mbox{for all\ } u\in V\}$.
In order to establish this we will use Lemma~\ref{LemmaB}.

Let $b\in B$ where $B$ is the Borel subgroup of all upper triangular matrices in $G=\GlnF$, so $b=(b_{ij})$ where $b_{ij}=0$ whenever $i>j$, and $b_{11}b_{22}b_{33}\ne0$.
Then $\boldsymbol\lambda b=\big(\frac{b_{22}b_{33}}{b_{11}}\big){\bf231}+\beta\big(\frac{b_{22}b_{33}}{b_{11}}\big){\bf321}+(1+\beta)\big(\frac{b_{23}b_{33}}{b_{11}}\big){\bf331}$.
Our first goal is to show that $\overline{\boldsymbol\lambda B}\cap\mathcal K_3=\{{\bf0}\}$.

It follows easily from the expression for $\boldsymbol\lambda b$ obtained above that the following polynomials all belong to ${\bf I}(\boldsymbol\lambda B)$, the vanishing ideal of the $B$-orbit $\boldsymbol\lambda B:$\\
$X_{321}-\beta X_{231}$, $X_{232}$, $X_{233}$, $X_{322}$, $X_{323}$ and $X_{11i}$, $X_{12i}$, $X_{13i}$, $X_{21i}$, $X_{22i}$, $X_{31i}$ for $1\le i\le 3$.

Now let $\boldsymbol\mu=(\mu_{ijk})\in(\overline{\boldsymbol\lambda B})\cap\mathcal K_3$.
Since $\boldsymbol\mu\in\mathcal K_3$, we must have $\mu_{iik}=0$ for $1\le i,k\le 3$.
Moreover, in view of the fact that in $\Theta^{-1}(\boldsymbol\mu)$ we have $[e_2+e_3,e_2+e_3]=0_V$, we also get that $\mu_{321}+\mu_{231}=0$.
Since $\boldsymbol\mu\in\overline{\boldsymbol\lambda B}$ and $X_{321}-\beta X_{231}\in{\bf I}(\boldsymbol\lambda B)$ we must also have that $\mu_{321}-\beta\mu_{231}=0$.
This gives $(1+\beta)\mu_{321}=0$ with $\beta\ne-1$, so $\mu_{321}=0=\mu_{231}$.
Finally, invoking the fact that $\ev_{\boldsymbol\mu}(f)=0$ for each of the remaining polynomials $f$ in ${\bf I}(\boldsymbol\lambda B)$ listed above we conclude that $\boldsymbol\mu={\bf0}$.
Hence $\overline{\boldsymbol\lambda B}\cap\mathcal K_3=\{{\bf0}\}$.
Remark~\ref{RemarkBorelSubgroup}(ii) with $U=\mathcal K_3-\{{\bf0}\}$ now ensures that $U\cap\overline{O(\boldsymbol\lambda)}=\varnothing$ and hence $\overline{O(\boldsymbol\lambda)}\cap\mathcal K_3=\{{\bf0}\}$ as required.
This observation will turn out to be useful in Section~\ref{SectionDegen3dim}.
\end{example}

\begin{remark}\label{Remark2.7}
(i) An argument involving the $B$-orbit may not only be used to rule out the possibility of degeneration as in Example~\ref{ExamplApplB}, but it may also be used in order to construct degenerations:
Keeping the hypothesis and notation of Example~\ref{ExamplApplB}, define $b(t)\in B$ for each $t\in\F-\{0\}$ by setting $b_{23}=1$ and $b_{22}=b_{33}=b_{11}=t$ (arbitrary values can be assigned to $b_{12}$ and $b_{13}$).
We then have that  $\boldsymbol\lambda b(t)=t\,{\bf231}+(\beta t)\,{\bf321}+(1+\beta)\,{\bf331}$, $t\in\F-\{0\}$.
Let $f\colon\F\to \boldsymbol\Lambda\colon$ $t\mapsto t\,{\bf231}+(\beta t)\,{\bf321}+(1+\beta)\,{\bf331}$, ($t\in\F$).
Also let $S=\F-\{0\}$.
Then $f$ is a continuous map in the Zariski topology satisfying $f(S)\subseteq O(\lambda)$.
Invoking Lemma~\ref{LemmaA} we get that $(1+\beta){\bf331}\,(=f(0))\in\overline{f(S)}\subseteq\overline{O(\boldsymbol\lambda)}$.
Our hypothesis that $1+\beta\ne0$ now ensures that ${\bf331}\in \overline{O(\boldsymbol\lambda)}$.

(ii) The arguments in Example~\ref{ExamplApplA} and in item~(i) of this remark  also provide a method of constructing 1-parameter contractions, if one is working in the metric topology with $\F=\mathbb C$ or $\mathbb R$, by letting $t\to0$.
\end{remark}

\begin{lemma}\label{LemmaC}
Let $1<r<n$ and let $\boldsymbol\lambda\,(=(\lambda_{ijk}))\in\boldsymbol\Lambda$ satisfy $\lambda_{ijk}=0$ whenever $1\le i,j\le r$ and $r+1\le k\le n$.
Define $\boldsymbol\mu\,(=(\mu_{ijk}))\in\boldsymbol\Lambda$ by
\[
\mu_{ijk}=\left\{
\begin{array}{ll}
\lambda_{ijk},& \mbox{if\ } (1\le i,j\le r \mbox{\ and\ } 1\le k\le n) \mbox{\ or\ } (1\le i\le r \mbox{\ and\ } r+1\le j,k\le n)\\ & \mbox{\qquad\qquad or\ } (1\le j\le r \mbox{\ and\ } r+1\le i,k\le n),\\
0,& \mbox{otherwise}.
\end{array}
\right.
\]
Then $\boldsymbol\mu\in\overline{O(\boldsymbol\lambda)}$.
\end{lemma}
\begin{proof}
Assume the hypothesis.
For each $t\in\F-\{0\}$, let $g(t)\in G$ be the diagonal matrix having coefficient 1 in the first $r$ entries and coefficient t in the last $n-r$ entries.
Also let $\boldsymbol\nu(t)\,(=(\nu_{ijk}(t)))\in\boldsymbol\Lambda$ be defined, for each $t\in\F$, by
\[
\nu_{ijk}(t)=\left\{
\begin{array}{ll}
\lambda_{ijk},& \mbox{if\ } (1\le i,j\le r \mbox{\ and\ } 1\le k\le n) \mbox{\ or\ } (1\le i\le r \mbox{\ and\ } r+1\le j,k\le n)\\ & \mbox{\qquad\qquad or\ } (1\le j\le r \mbox{\ and\ } r+1\le i,k\le n),\\
t^2\,\lambda_{ijk},& \mbox{if\ } r+1\le i,j\le n \mbox{\ and\ } 1\le k\le r,\\
t\,\lambda_{ijk},& \mbox{otherwise}.
\end{array}
\right.
\]
Then $\boldsymbol\nu(t)=\boldsymbol\lambda\, g(t)$ for each $t\in\F-\{0\}$.
We now define $f\colon\F\to\boldsymbol\Lambda:$ $t\mapsto\boldsymbol\nu(t)$, ($t\in\F$).
Then $f$ is a continuous function in the Zariski topology.
Moreover, by setting $S=\F-\{0\}$, we see that $f(S)\subseteq O(\boldsymbol\lambda)$.
Finally, invoking Lemma~\ref{LemmaA} we get that $\boldsymbol\mu\,(=f(0))\in \overline{f(S)}\subseteq\overline{O(\boldsymbol\lambda)}$.
\end{proof}

\begin{remark}\label{Remark2.9}
Keeping the notation and hypothesis of Lemma~\ref{LemmaC}, we see that $\mathfrak b=\F$-span$(e_1,\ldots,e_r)$ is in fact a subalgebra of $\Theta^{-1}(\boldsymbol\lambda)$.
Moreover, $\Theta^{-1}(\boldsymbol\mu)$ is a `semi-direct' sum of the algebra $\mathfrak b$ with an Abelian ideal of dimension $n-r$.
\end{remark}

\begin{example}\label{ExamplApplC}
(An application of Lemma~\ref{LemmaC}.)\\
Let $n=3$ and let $\boldsymbol\lambda={\bf121}+{\bf211}+{\bf222}+{\bf323}\in\boldsymbol\Lambda$.
We aim to show that ${\bf213}\in \overline{O(\boldsymbol\lambda)}$.
It will be convenient to consider the structure vector $\boldsymbol\nu=\boldsymbol\lambda g$ where
$g=${\footnotesize$\begin{pmatrix}
1&0&0\\
0&1&0\\
1&0&-1
\end{pmatrix}$}$\in\GlF
$.
We have $\boldsymbol\nu={\bf121}+{\bf211}+{\bf213}+{\bf222}+{\bf323}\in\boldsymbol\Lambda$.
Invoking Lemma~\ref{LemmaC} with $r=1$ we get that ${\bf213}\in\overline{O(\boldsymbol\nu)}=\overline{O(\boldsymbol\lambda)}$ (since $O(\boldsymbol\nu)=O(\boldsymbol\lambda)$).
\end{example}

Finally for this section we include the following preparatory lemma which we will need in Section~\ref{SectionDegen3dim}.
\begin{lemma}\label{Lemma5.2}
Let $\kappa\in\mathbb C-\{-2,2\}$ and let $\boldsymbol\lambda={\bf221}+{\bf331}+\kappa({\bf321})\in\boldsymbol\Lambda$.
Also let $\alpha\in\mathbb C$  be a root of the polynomial $x^2+\kappa x+1\in\mathbb C[x]$.
Then:

(i) $\alpha^2\ne1$, (ii) $\kappa+2\alpha\ne0$, (iii) $\frac{\alpha(\kappa\alpha+2)}{\kappa+2\alpha}=-\alpha^2$, and (iv) ${\bf231}+\,-\alpha^2({\bf321})\in O(\boldsymbol\lambda)$.
\end{lemma}

\begin{proof}
Assume the hypothesis.

(i) Since $\alpha$ is a root of $x^2+\kappa x+1$, we get $x^2+\kappa x+1=(x-\alpha)(x-\beta)$ for some $\beta\in\mathbb C$ with $\alpha\beta=1$.
In particular $\beta$ is also a root of $x^2+\kappa x+1$ and $\alpha\ne0$, $\beta\ne0$.
Moreover, $\alpha\ne\beta$ since the roots of the polynomial $x^2+\kappa x+1$ over $\mathbb C$ are given by $\frac12(-\kappa\pm\sqrt{\kappa^2-4})$ and $\kappa^2\ne4$ by assumption.
If $\alpha^2=1$, we would get $\alpha=\alpha(\alpha\beta)=\alpha^2\beta=\beta$ which is a contradiction.
We conclude that $\alpha^2\ne1$.

(ii) From $\alpha^2+\kappa\alpha+1=0$ we get $\kappa=\alpha^{-1}(-1-\alpha^2)$, (recall that $\alpha\ne0$).
Hence $\kappa+2\alpha=\alpha^{-1}(-1-\alpha^2+2\alpha^2)=\alpha^{-1}(-1+\alpha^2)\ne0$, since $\alpha^2\ne1$ from item~(i) of this lemma.

(iii) From item (ii) we get, $\alpha(\kappa\alpha+2)=\alpha(-1-\alpha^2+2)=\alpha(1-\alpha^2)$.
Hence,
\[
\frac{\alpha(\kappa\alpha+2)}{\kappa+2\alpha}=\frac{\alpha(1-\alpha^2)}{\alpha^{-1}(-1+\alpha^2)}=-\alpha^2.
\]

(iv) Let
$g=${\footnotesize$\begin{pmatrix}
\kappa+2\alpha&0&0\\
0&\alpha&1\\
0&1&\alpha
\end{pmatrix}$}$\in M_3(\mathbb C)$.
Clearly $g\in\GlC$ since $\det(g)=(\kappa+2\alpha)(\alpha^2-1)\ne0$.
Then $\boldsymbol\lambda g={\bf231}+\frac{\alpha(\kappa\alpha+2)}{\kappa+2\alpha}({\bf321})={\bf231}+\,-\alpha^2({\bf321})$.
Hence, ${\bf231}+\,-\alpha^2({\bf321})\in O(\boldsymbol\lambda)$ as required.
\end{proof}

\section{Necessary conditions for degeneration}\label{SectionNecCond}

Necessary conditions in the study of degenerations have been used extensively by many authors.
In this section we recall certain necessary conditions for degeneration which we will need later~on.

\begin{definition}
Let $\mathfrak g=(V,\,[,])\in\boldsymbol A$.
We define the following $\F$-vector spaces:

(i) ${\rm ann}_R\,\mathfrak g=\{c\in V\colon [a,c]=0_V$ for all $a\in V\}$, (the right annihilator of $\mathfrak g$),

(ii) ${\rm ann}_L\,\mathfrak g=\{c\in V\colon [c,a]=0_V$ for all $a\in V\}$, (the left annihilator of $\mathfrak g$),

(iii) ${\rm Der}_{(\alpha,\beta,\gamma)}\mathfrak g=\{\phi\in{\rm End}_\F\,\mathfrak g\colon \alpha\phi[u,v]=\beta[\phi(u),v]+\gamma[u,\phi(v)]$ for all $u,v\in V\}$, for each ordered triple $(\alpha,\beta,\gamma)\in\F^3$, and

(iv) ${\rm Der}\mathfrak g={\rm Der}_{(1,1,1)}\mathfrak g$, (the algebra of derivations of $\mathfrak g$).
\end{definition}

In the following proposition we collect some well-known facts regarding degenerations.
We supply the proofs for completeness.

\begin{proposition}
Let $\mathfrak g, \mathfrak h\in\boldsymbol A$ and suppose $\mathfrak g$ degenerates to $\mathfrak h$.
Then

(i) $\dim_{\F}{\rm ann}_R\,\mathfrak g\le\dim_{\F}{\rm ann}_R\,\mathfrak h$.

(ii) $\dim_{\F}{\rm ann}_L\,\mathfrak g\le\dim_{\F}{\rm ann}_L\,\mathfrak h$.

(iii) $\dim_{\F}{\rm Der}_{(\alpha,\beta,\gamma)}\mathfrak g\le\dim_{\F}{\rm Der}_{(\alpha,\beta,\gamma)}\mathfrak h$.

(iv) If $\F=\mathbb C$ and $\mathfrak g$ properly degenerates to $\mathfrak h$, then $\dim_{\F}{\rm Der}\,\mathfrak g<\dim_{\F}{\rm Der}\,\mathfrak h$.

\end{proposition}
\begin{proof}
(i) For each $\boldsymbol\lambda=(\lambda_{ijk})\in\boldsymbol\Lambda$ let $\tilde c(\boldsymbol\lambda)\in\F^{n\times n^2}$ be the $(n\times n^2)$-matrix over $\F$ whose columns are precisely all the vectors of the form $(\lambda_{i1j}, \lambda_{i2j},\ldots, \lambda_{inj})^{tr}$ for $1\le i,j\le n$ (in some fixed order).
Also define $\psi(\boldsymbol\lambda)\colon \F^{1\times n}\to\F^{n\times n^2}:$ $\hat{\alpha}\mapsto\hat\alpha\tilde c(\boldsymbol\lambda)$, $\hat\alpha\in\F^{1\times n}$ (the space of $(1\times n)$-matrices over $\F$).
Then $\psi(\boldsymbol\lambda)$ is an $\F$-linear map and it is easy to see that for $\alpha_1, \ldots, \alpha_n\in\F$ we have $\alpha_1e_1+\ldots+\alpha_ne_n\in{\rm ann}_R\Theta^{-1}(\boldsymbol\lambda)$ if, and only if, $(\alpha_1,\ldots,\alpha_n)\in\ker\psi(\boldsymbol\lambda)$.
Hence, $\dim_\F({\rm ann}_R\Theta^{-1}(\boldsymbol\lambda))=\dim_{\F}\ker\psi(\boldsymbol\lambda)=n-\dim_\F{\rm im}\psi(\boldsymbol\lambda)=n-{\rm rank}\,\tilde c(\boldsymbol\lambda)$.
It follows that ${\rm rank}\,\tilde c(\boldsymbol\lambda)={\rm rank}\,\tilde c(\boldsymbol\mu)$ whenever $\boldsymbol\mu\in O(\boldsymbol\lambda)$ since the algebras $\Theta^{-1}(\boldsymbol\lambda)$ and $\Theta^{-1}(\boldsymbol\mu)$ are isomorphic (so $\dim_\F{\rm ann}_R\Theta^{-1}(\boldsymbol\lambda)=\dim_\F{\rm ann}_R\Theta^{-1}(\boldsymbol\mu)$).
We thus have that for each non-negative integer $t$ the set $\{\boldsymbol\nu\in\boldsymbol \Lambda\colon {\rm rank}\,\tilde c(\boldsymbol\nu)\le t\}$ is a Zariski-closed set in $\boldsymbol\Lambda$ which is also a union of orbits (compare, for example, with~\cite[Remark~3.15]{IvanovaPallikaros2019}).
We conclude that ${\rm rank}\,\tilde c(\boldsymbol\nu)\le {\rm rank}\,\tilde c(\boldsymbol\lambda)$ whenever $\boldsymbol\nu\in\overline{O(\boldsymbol\lambda)}$.
This leads to $\dim_\F{\rm ann}_R\Theta^{-1}(\boldsymbol\nu)\ge\dim_\F{\rm ann}_R\Theta^{-1}(\boldsymbol\lambda)$ whenever $\boldsymbol\nu\in\overline{O(\boldsymbol\lambda)}$ as required.

(ii) We can use same argument as in item (i) of this proposition, but in the place of the matrix $\tilde c(\boldsymbol\lambda)$ we now consider matrix $\tilde a(\boldsymbol\lambda)\in\F^{n\times n^2}$ where the columns of $\tilde a(\boldsymbol\lambda)$ are precisely the vectors of the form $(\lambda_{1ij},\lambda_{2ij},\ldots,\lambda_{nij})^{tr}$ for $1\le i,j\le n$.
(See also the proof of~\cite[Lemma~3.16]{IvanovaPallikaros2019}.)

(iii) Let $\boldsymbol\lambda\,(=(\lambda_{ijk}))\in\boldsymbol\Lambda$ and let $\phi\in{\rm End}_{\F}V$.
Also let $\alpha$, $\beta$ and $\gamma$ be (not necessarily distinct) elements of $\F$.
For $1\le r,s\le n$ we denote by $\phi_{(r,s)}\,(\in\F)$ the coefficient of $e_r$ when we express $\phi(e_s)$ as an $\F$-linear combination of the elements of the standard basis $(e_1,\ldots,e_n)$ of $V$.
Suppose further that the algebra $\Theta^{-1}(\boldsymbol\lambda)$ has multiplication given by $[,]$.
Recall that in Section~\ref{SectionPrelim} we have imposed an ordering on the $n^3$ ordered triples $(i,j,k)$, $1\le i,j,k\le n$.
For the discussion that follows it will also be convenient to impose an ordering on the $n^2$ ordered pairs $(r,s)$, $1\le r,s\le n$.
Via this ordering we can consider the row-matrix $\hat\phi\in\F^{1\times n^2}$ whose coefficients are the $\phi_{(r,s)}$ for $1\le r,s\le n$.
The map $\phi\mapsto\hat\phi$ from ${\rm End}_{\F}V$ to $\F^{1\times n^2}$ is then an isomorphism of $\F$-spaces.

By expanding (for $1\le i,j\le n$) the expression $\alpha\phi([e_i,e_j])-\beta[\phi(e_i),e_j]-\gamma[e_i,\phi(e_j)]\,(\in V)$ as a linear combination of the elements of the standard basis $(e_1,\ldots,e_n)$ of $V$, we can see that the coefficient of $e_k$ ($1\le k\le n$) in this expression is given by $\sum_{r,s=1}^n\phi_{(r,s)}\delta_{(r,s)}^{(i,j,k)}\,(\in\F)$ for some coefficients $\delta_{(r,s)}^{(i,j,k)}\in\F$ which depend on $i,j,k,r$ and $s$.
The coefficients $\delta_{(r,s)}^{(i,j,k)}$ are in fact expressions of the form $\sum\xi_{abc}\lambda_{abc}$ with the only allowed values of the elements $\xi_{abc}$ of $\F$ coming from the list $0$, $\alpha$, $\beta$, $\gamma$.

[For example, when $n=3$, we get $\delta_{(1,1)}^{(1,2,1)}=\alpha\lambda_{121}-\beta\lambda_{121}$ and $\delta_{(1,2)}^{(1,2,1)}=\alpha\lambda_{122}-\gamma\lambda_{111}$.]

For each $\boldsymbol\lambda\in\boldsymbol\Lambda$ we can now define matrix $\tilde d(\boldsymbol\lambda)\in\F^{n^2\times n^3}$ as the matrix having the $n^5$ elements $\delta_{(r,s)}^{(i,j,k)}$ as its coefficients: the column-index (resp., row-index) of the coefficient $\delta_{(r,s)}^{(i,j,k)}$ in the matrix $\tilde d(\boldsymbol\lambda)$ is given by the ordering we have fixed on the triples $(i,j,k)$ (resp., the pairs $(r,s)$).
In particular, the coefficients in the $(i,j,k)$-column of $\tilde d(\boldsymbol\lambda)$ are the $\delta_{(r,s)}^{(i,j,k)}$ for $1\le r,s\le n$.

It is then easy to observe that $\phi\in{\rm Der}_{(\alpha,\beta,\gamma)}\Theta^{-1}(\boldsymbol\lambda)$ if, and only if, $\hat\phi\, \tilde d(\boldsymbol\lambda)={\bf0}_{1\times n^3}$, the $(1\times n^3)$ zero matrix.
It follows that $\dim{\rm Der}_{(\alpha,\beta,\gamma)}\Theta^{-1}(\boldsymbol\lambda)=n^2-{\rm rank}\, \tilde d(\boldsymbol\lambda)$ (since this equals to the dimension of the kernel of the linear map $\hat u\mapsto \hat u\, \tilde d(\boldsymbol\lambda)$ from $\F^{1\times n^2}$ to $\F^{1\times n^3}$). 
Comparing, for example, with~\cite[Result~3.13 and Remark~3.15]{IvanovaPallikaros2019} we see that the set $\{\boldsymbol\mu\in\boldsymbol\Lambda\colon {\rm rank}\,\tilde d(\boldsymbol\mu)\le {\rm rank}\,\tilde d(\boldsymbol\lambda)\}$ is Zariski-closed.
Now ${\rm rank}\,\tilde d(\boldsymbol\lambda)={\rm rank}\,\tilde d(\boldsymbol\lambda')$ whenever $\boldsymbol\lambda'\in O(\boldsymbol\lambda)$, since $\Theta^{-1}(\boldsymbol\lambda)\simeq\Theta^{-1}(\boldsymbol\lambda')$.
Hence, $O(\boldsymbol\lambda)\subseteq\{\boldsymbol\mu\in\boldsymbol\Lambda\colon {\rm rank}\,\tilde d(\boldsymbol\mu)\le {\rm rank}\,\tilde d(\boldsymbol\lambda)\}$.
It follows that $\dim{\rm Der}_{(\alpha,\beta,\gamma)}\Theta^{-1}(\boldsymbol\mu)\ge\dim{\rm Der}_{(\alpha,\beta,\gamma)}\Theta^{-1}(\boldsymbol\lambda)$ whenever $\boldsymbol\mu\in\overline{O(\boldsymbol\lambda)}$.

Item (iv) follows, for example, by combining~\cite[Propositions~1.5.2 and~2.5.3]{Geck2003} and~\cite[Proposition~p.60 and Corollary~p.88]{Humphrays1991}.
See also~\cite[Example~2, p.~23]{OnishchikVinberg90}.
\end{proof}

Clearly, if for some $\boldsymbol\lambda\in\boldsymbol\Lambda$ we have  $O(\boldsymbol\lambda)\subseteq S$, with $S$ an algebraic subset of $\boldsymbol\Lambda$, then $\boldsymbol\mu\in S$ for all $\boldsymbol\mu\in\overline{O(\boldsymbol\lambda)}$.
In the next section we recall from~\cite{IvanovaPallikaros2019}, \cite{PallikarosWard2020} some particular algebraic subsets of $\boldsymbol\Lambda$ which played a key role in those papers and which will also play some part in the present paper.

\section{Certain algebraic subsets of $\boldsymbol\Lambda$ and the maps $\phi_1$, $\phi_2$}\label{SectionPhi1Phi2}

Following the notation in~\cite{PallikarosWard2020} we define the subsets $\mathcal K$, $\mathcal C$, $\mathcal M^*$ and $\mathcal M^{**}$ of~$\boldsymbol\Lambda$ by
\begin{gather*}
\Theta^{-1}(\mathcal K)=\{\mathfrak g=(V,[,])\in\boldsymbol A\colon [u,u]=0_V\ \mbox{for all\ } u\in V\},\\
\Theta^{-1}(\mathcal C)=\{\mathfrak g=(V,[,])\in\boldsymbol A\colon [u,v]=[v,u]\ \mbox{for all\ } u,v\in V\},\\
\Theta^{-1}(\mathcal M^*)=\{\mathfrak g=(V,[,])\in\boldsymbol A\colon [u,v]\in\F\mbox{-span}(u,v)\ \mbox{for all\ } u,v\in V\}, \ \mbox{and}\\
\Theta^{-1}(\mathcal M^{**})=\{\mathfrak g=(V,[,])\in\boldsymbol A\colon [u,u]\in\F\mbox{-span}(u)\ \mbox{for all\ } u\in V\}.
\end{gather*}
Then (see~\cite[Sections~3.1, 5.1 (Eq. (7)) and~6.1 (Eq. (12))]{PallikarosWard2020}) we have
\begin{gather*}
\mathcal K=\{\boldsymbol\lambda=(\lambda_{ijk})\in\boldsymbol\Lambda\colon \lambda_{iii}=0,\ \lambda_{iij}=0,\ \lambda_{ijk}+\lambda_{jik}=0\ \mbox{and\ } \lambda_{iji}+\lambda_{jii}=0\},\\
\mathcal C=\{\boldsymbol\lambda=(\lambda_{ijk})\in\boldsymbol\Lambda\colon \lambda_{ijj}=\lambda_{jij}\ \mbox{and\ } \lambda_{ijk}=\lambda_{jik}\},\\
\mathcal M^*=\{\boldsymbol\lambda=(\lambda_{ijk})\in\boldsymbol\Lambda\colon \lambda_{iij}=0,\ \lambda_{ijk}=0,\ \lambda_{ijj}=\lambda_{ikk},\ \lambda_{jij}=\lambda_{kik}\ \mbox{and}\\
\phantom{\mathcal M^*=\{}{} \lambda_{iii}=\lambda_{ijj}+\lambda_{jij}\}, \ \mbox{and}\\
\mathcal M^{**}=\{\boldsymbol\lambda=(\lambda_{ijk})\in\boldsymbol\Lambda\colon \lambda_{ijk}+\lambda_{jik}=0,\ \lambda_{iij}=0\ \mbox{and\ } \lambda_{iii}=\lambda_{ijj}+\lambda_{jij}\},
\end{gather*}
where the following convention is in force for the description of the last four sets:
Different letters in the subscripts for the components of a structure vector represent different numerical values, but all such choices of subscripts are allowed.
In particular, $\mathcal K$, $\mathcal C$, $\mathcal M^*$ and $\mathcal M^{**}$ are Zariski-closed subsets of $\boldsymbol\Lambda$.
Moreover, these sets are all unions of orbits with respect to the action of $G$ on $\boldsymbol\Lambda$ we are considering.

We also define the subsets $\mathcal B$ and $\mathcal T$ of $\boldsymbol\Lambda$ by
\begin{gather*}
\Theta^{-1}(\mathcal B)=\{\mathfrak g=(V,[,])\in\boldsymbol A\colon [[u,v],w]=0_V\ \mbox{for all\ } u,v,w\in V\}, \ \mbox{and}\\
\Theta^{-1}(\mathcal T)=\{\mathfrak g=(V,[,])\in\boldsymbol A\colon {\rm trace}\,{\rm ad}_u=0\ \mbox{for all\ } u\in V\},
\end{gather*}
where ${\rm ad}_u\colon V\to V:$ $v\mapsto [u,v]$, ($v\in V$) is the adjoint map.
Then, (see for example~\cite[Remark~2.7 and Remark~4.12]{IvanovaPallikaros2019}), we have that $\mathcal B=\{\boldsymbol\lambda=(\lambda_{ijk})\in\boldsymbol\Lambda\colon \sum_{l=1}^n\lambda_{ijl}\lambda_{lkm}=0$ for $1\le i,j,k,m\le n\}$ and $\mathcal T=\{\boldsymbol\lambda=(\lambda_{ijk})\in\boldsymbol\Lambda\colon \sum_{j=1}^n\lambda_{ijj}=0$ for $1\le i\le n\}$.
In particular, both $\mathcal B$ and $\mathcal T$ are algebraic subsets of $\boldsymbol\Lambda$ which are also unions of orbits.

Finally, we introduce the subsets $\mathcal A^s$, $\mathcal L$ and $\mathcal J$ of $\boldsymbol\Lambda$ by
\begin{gather*}
\Theta^{-1}(\mathcal A^s)=\{\mathfrak g=(V,[,])\in\boldsymbol A\colon [[u,v],w]=[u,[v,w]]\ \mbox{for all\ } u,v,w\in V\},\\
\Theta^{-1}(\mathcal L)=\{\mathfrak g=(V,[,])\in\Theta^{-1}(\mathcal K)\colon [[u,v],w]+[[v,w],u]+[w,u],v]=0_V\ \mbox{for all\ } u,v,w\in V\}, \\
\Theta^{-1}(\mathcal J)=\{\mathfrak g=(V,[,])\in\Theta^{-1}(\mathcal C)\colon [[[u,u],v],u]=[[u,u],[v,u]]\ \mbox{for all\ } u,v\in V\}
\end{gather*}
which correspond respectively to the sets of Associative, Lie and Jordan algebra structures in~$\boldsymbol A$.

The sets $\mathcal A^s$, $\mathcal L$ and $\mathcal J$ of $\boldsymbol A$ are also algebraic and consist of unions of orbits.
For defining conditions of these sets via polynomial equations see for example~\cite[Proposition~1, p.4]{Jacobson1962} for the sets $\mathcal A^s$ and $\mathcal L$, and~\cite[Section~2]{KashubaShestakov2007} for~$\mathcal J$.

These last three sets will be of central importance in this paper as our goal is to determine the degeneration picture in $\mathcal A^s$ for the special case $n=3$ and $\F=\mathbb C$ via the degeneration pictures in $\mathcal L$ and $\mathcal J$.
In order to achieve this we will need first to introduce certain maps $\phi_1$, $\phi_2$ satisfying $\phi_1(\mathcal A^s)\subseteq\mathcal L$ and $\phi_2(\mathcal A^s)\subseteq\mathcal J$.
For this we need ${\rm char}\,\F\ne2$ and we therefore make this assumption for the rest of this section.

\bigskip

For an algebra $\mathfrak g=(V,\,[,])\in\boldsymbol A$, the opposite algebra $\widetilde{\mathfrak{g}}$ has product $\widetilde{[\,,\,]}$
defined by $\widetilde{[u,v]}=[v,u]$ for all $u,v\in V$.
If $\Theta(\mathfrak{g})=\boldsymbol{%
\lambda}$ with $\boldsymbol{\lambda}=(\lambda_{ijk})$, we will write $\Theta(\widetilde{%
\mathfrak{g}})=\widetilde{\boldsymbol{\lambda}}$ with $\widetilde{\boldsymbol{\lambda%
}}=(\widetilde{\lambda}_{ijk})$. Clearly $(\widetilde{\tilde{\boldsymbol{\lambda}}})=%
\boldsymbol{\lambda}$ and $\widetilde\lambda_{ijk}=\lambda_{jik}$ for all $i,j,k$.
In~\cite[Lemma~3.2]{PallikarosWard2020} it is shown that  $(\widetilde{\boldsymbol{\lambda}})g=(\widetilde{\boldsymbol{%
\lambda }g})$ for all $\boldsymbol{\lambda}\in\boldsymbol{\Lambda}$ and for
all $g\in G$.

We now define the maps $\phi$, $\phi_1$ and $\phi_2\colon\boldsymbol\Lambda\to\boldsymbol\Lambda$, respectively, by $\boldsymbol\lambda\mapsto\tilde{\boldsymbol\lambda}$, $\boldsymbol\lambda\mapsto\frac12(\boldsymbol\lambda-\tilde{\boldsymbol\lambda})$ and $\boldsymbol\lambda\mapsto\frac12(\boldsymbol\lambda+\tilde{\boldsymbol\lambda})$ for $\boldsymbol\lambda\in \boldsymbol\Lambda$.
Clearly the maps $\phi$, $\phi_1$ and $\phi_2$ are regular and hence continuous in the Zariski topology.
Moreover, we have $\phi(\mathcal A^s)=\mathcal A^s$, $\phi(\mathcal L)=\mathcal L$, $\phi(\mathcal J)=\mathcal J$, $\phi_1(\mathcal A^s)\subseteq\mathcal L$ and $\phi_2(\mathcal A^s)\subseteq\mathcal J$.

We collect some further observations regarding the maps $\phi$, $\phi_1$ and $\phi_2$ in the following remark.

\begin{remark}\label{RemarkObservations}
Let $\psi$ be any of the maps $\phi$, $\phi_1$ or $\phi_2$.
Also let $\boldsymbol\lambda\in\boldsymbol\Lambda$.

(i) It follows from the fact that $(\widetilde{\boldsymbol{\lambda}})g=(\widetilde{\boldsymbol{\lambda }g})$ for any $g\in G$, and the linearity of the $G$-action we are considering, that $\psi$ is an $\F G$-module homomorphism from $\boldsymbol\Lambda$ to $\boldsymbol\Lambda$, see~\cite[Lemma~3.2]{PallikarosWard2020}.
In particular, $\psi(O(\boldsymbol\lambda))=\{\psi(\boldsymbol\lambda g)\colon g\in G\}=\{\psi(\boldsymbol\lambda)g\colon g\in G\}=O(\psi(\boldsymbol\lambda))$.

(ii) If $S$ is subset of $\boldsymbol\Lambda$ which is a union of $G$-orbits, then the subset $\psi^{-1}(S)$ of $\boldsymbol\Lambda$ is also a union of $G$-orbits since for $g\in G$ and $\boldsymbol\nu\in\psi^{-1}(S)$ we have $\psi(\boldsymbol\nu g)=\psi(\boldsymbol\nu)g\in S$.

(iii) From the continuity of $\psi$ we get $\psi(\overline{O(\boldsymbol\lambda)})\subseteq\overline{O(\psi(\boldsymbol\lambda))}\,(=\overline{\psi(O(\boldsymbol\lambda))})$.

(iv) Since $\phi$ is invertible with $\phi^{-1}=\phi$, we must have $\phi\,(\overline{O(\boldsymbol\lambda)})=\overline{O(\phi(\boldsymbol\lambda))}\,( =\overline{ O(\tilde{\boldsymbol\lambda} ) } )$.

(v)  If $\F$ is algebraically closed and $\tilde{\boldsymbol\lambda}\not\in{O(\boldsymbol\lambda)}$, then $\tilde{\boldsymbol\lambda}\not\in\overline{O(\boldsymbol\lambda)}$ (and $\boldsymbol\lambda\not\in\overline{O(\tilde{\boldsymbol\lambda})})$.
To see this, note that if $\tilde{\boldsymbol\lambda}\in\overline{O(\boldsymbol\lambda)}-O(\boldsymbol\lambda)$, then $\boldsymbol\lambda\,(=\phi(\tilde{\boldsymbol\lambda}))\in\phi(\overline{O(\boldsymbol\lambda)})=\overline{O(\tilde{\boldsymbol\lambda})}$, which is impossible as orbits are open in their closure --- see for example~\cite[Proposition 2.5.2(a)]{Geck2003}.
\end{remark}

\section{
Three-dimensional complex Associative, Lie and Jordan\\ algebras}\label{Section3dimAlg}

We fix $\F=\mathbb C$ and $n=3$ for the rest of the paper.
We also denote by $\mathbb K$ the subset\\ $\{x+yi\colon x,y\in\mathbb R$ with $x>0$ or ($x=0$ and $y>0)\}$ of $\mathbb C$.

{\begin{center}
\footnotesize
\refstepcounter{tbn}\label{Table3dimAssocClass}

\begin{longtable}{|c|p{11cm}|l|l|}
\hline
$\mathfrak a$ & Non-zero commutation relations relative to the standard basis of $V$ & $\phi_1$ & $\phi_2$\\
\hline
$\mathfrak a_0$ & --- &$\mathfrak a_0$ & $\mathfrak a_0$\\
\hline
$\mathfrak l_1$  & $e_2e_3=-e_3e_2=e_1$ & $\mathfrak l_1$ & $\mathfrak a_0$ \\
\hline
\hline
$\mathfrak c_{1}$  & $e_3e_3=e_1$ & $\mathfrak a_0$ & $\mathfrak c_{1}$\\
\hline
$\mathfrak c_{2}$ & $e_1e_1=e_1$, $e_1e_2=e_2e_1=e_2$, $e_1e_3=e_3e_1=e_3$  & $\mathfrak a_0$ & $\mathfrak c_{2}$ \\
\hline
$\mathfrak c_{3}$  & $e_2e_2=e_1$, $e_3e_3=e_1$  & $\mathfrak a_0$ & $\mathfrak c_{3}$\\
\hline
$\mathfrak c_{4}$  & $e_1e_1=e_1$  & $\mathfrak a_0$ & $\mathfrak c_{4}$\\
\hline
$\mathfrak c_{5}$  & $e_1e_1=e_2$, $e_1e_2=e_2e_1=e_3$  & $\mathfrak a_0$ & $\mathfrak c_{5}$ \\
\hline
$\mathfrak c_{6}$ & $e_1e_1=e_1$, $e_1e_2=e_2e_1=e_2$, $e_1e_3=e_3e_1=e_3$, $e_2e_2=e_3$  & $\mathfrak a_0$ & $\mathfrak c_{6}$ \\
\hline
$\mathfrak c_{7}$   & $e_1e_1=e_1$, $e_1e_2=e_2e_1=e_2$  & $\mathfrak a_0$ & $\mathfrak c_{7}$\\
\hline
$\mathfrak c_{8}$   & $e_1e_1=e_1$, $e_2e_2=e_3$  & $\mathfrak a_0$ & $\mathfrak c_{8}$ \\
\hline
$\mathfrak c_{9}$  & $e_1e_1=e_1$, $e_2e_2=e_2$, $e_1e_3=e_3e_1=e_3$  & $\mathfrak a_0$ & $\mathfrak c_{9}$ \\
\hline
$\mathfrak c_{10}$  & $e_1e_1=e_1$, $e_2e_2=e_2$  & $\mathfrak a_0$ & $\mathfrak c_{10}$ \\
\hline
$\mathfrak c_{11}$  & $e_1e_1=e_1$, $e_2e_2=e_2$, $e_3e_3=e_3$  & $\mathfrak a_0$ & $\mathfrak c_{11}$\\
\hline
\hline
$\mathfrak a_{1}$  & $e_1e_1=e_1$, $e_1e_2=e_2e_1=e_2$, $e_1e_3=e_3e_1=e_3$, $e_2e_3=-e_3e_2=e_2$, $e_3^2=e_1$  & $\mathfrak m_4$& $J_2$ \\
\hline
$\mathfrak a_{2}$  & $e_2e_3=e_1$  & $\mathfrak l_1$ & $\mathfrak c_{3}$\\
\hline
$\mathfrak a_{3}(\kappa)$ & $e_2e_2=e_1$, $e_3e_2=\kappa e_1$, $e_3e_3=e_1$, {\tiny where $\kappa\in\mathbb K$ (different values of $\kappa\in\mathbb K$ correspond to non-isomorphic algebras)}  & $\mathfrak l_1$ & $\mathfrak c_{3}$ \\
\hline
$\mathfrak a_{4}$  & $e_3e_1=e_1$, $e_3e_2=e_2$, $e_3e_3=e_3$  & $\mathfrak m_2$ & $J_5$ \\
\hline
$\mathfrak a_{5}$  & $e_1e_3=e_1$, $e_2e_3=e_2$, $e_3e_3=e_3$  & $\mathfrak m_2$ & $J_5$ \\
\hline
$\mathfrak a_{6}$  & $e_2e_1=e_1$, $e_2e_2=e_2$, $e_3e_2=e_3$  & $\mathfrak m_3$ & $J_5$ \\
\hline
$\mathfrak a_{7}$ & $e_1e_2=e_2e_1=e_1$, $e_2e_2=e_2$, $e_3e_2=e_3$  & $\mathfrak m_4$ & $J_4$ \\
\hline
$\mathfrak a_{8}$ & $e_1e_2=e_2e_1=e_1$, $e_2e_2=e_2$, $e_2e_3=e_3$  & $\mathfrak m_4$ & $J_4$ \\
\hline
$\mathfrak a_{9}$ & $e_3e_2=e_2$, $e_3e_3=e_3$  & $\mathfrak m_4$ & $J_8$ \\
\hline
$\mathfrak a_{10}$ & $e_2e_3=e_2$, $e_3e_3=e_3$  & $\mathfrak m_4$ & $J_8$ \\
\hline
$\mathfrak a_{11}$  & $e_1e_1=e_1$, $e_3e_2=e_2$, $e_3e_3=e_3$  & $\mathfrak m_4$  & $J_3$ \\
\hline
$\mathfrak a_{12}$  & $e_1e_1=e_1$, $e_2e_3=e_2$, $e_3e_3=e_3$  & $\mathfrak m_4$  & $J_3$ \\
\hline
\end{longtable}

Table~\ref{Table3dimAssocClass}. Non-isomorphic 3-dimensional complex associative algebras.

\end{center}}

{\bf Comments on Table~\ref{Table3dimAssocClass}:}

(i) In the first column of Table~\ref{Table3dimAssocClass} we list a complete set of non-isomorphic 3-dimensional complex associative algebras based on the classification obtained in~\cite{KobayashiShirayanagiTsukada2021}.

\medskip
(ii) Entry $\mathfrak g$ (with $\mathfrak g\in\boldsymbol A$) in the column headed $\phi_1$ (resp., $\phi_2$) and in the row corresponding to algebra $\mathfrak a\in\Theta^{-1}(\mathcal A^s)$ means that $\phi_1(\Theta(\mathfrak a))\in O(\Theta(\mathfrak g))$ (resp., $\phi_2(\Theta(\mathfrak a))\in O(\Theta(\mathfrak g))$).

\medskip
(iii) A complete list of non-isomorphic algebras $\mathfrak g\in\boldsymbol A$ with $\Theta(\mathfrak g)$ lying in the set $\phi_1(\mathcal A^s)$ consists of the two associative Lie algebras (the Abelian algebra $\mathfrak a_0$ and the Heisenberg algebra $\mathfrak l_1$) and the following three non-associative Lie algebras (only non-zero commutation relations are listed):

$\mathfrak m_2$: $e_1e_3=-e_3e_1=e_1$, $e_2e_3=-e_3e_2=e_2$

$\mathfrak m_3$: $e_1e_3=-e_3e_1=e_1$, $e_2e_3=-e_3e_2=-e_2$

$\mathfrak m_4$: $e_1e_2=-e_2e_1=e_1$

\medskip

In Picture~\ref{Picture3DimLieChild} below we include all possible degenerations inside the set $\phi_1(\mathcal A^s)\,(\subseteq \mathcal L)$ as these were determined in~\cite{Agaoka1999,NesterenkoPopovych2006}.
In particular, the results in~\cite{Agaoka1999,NesterenkoPopovych2006} show that the set $\phi_1(\mathcal A^s)$ is an algebraic subset of $\boldsymbol\Lambda$.
We remark here that although the arguments in~\cite{Agaoka1999,NesterenkoPopovych2006} are made with respect to the metric topology (as the authors are interested in contractions of Lie algebras), they carry out easily to arguments in the Zariski topology (compare with Example~\ref{ExamplApplA} and Remark~\ref{Remark2.7}(ii) of the present paper).
In fact, it was shown in~\cite{GrunewaldOHalloran1988a} that the closures in the Zariski topology and in the standard topology of the orbit of a point of an affine variety over~$\mathbb C$ under the action of an algebraic group coincide.

\begin{center}
\begin{tikzpicture}
\matrix [column sep=7mm, row sep=5mm] {
& \node (g34)  {$\mathfrak m_3$}; &  &  \node (g21g1)  {$\mathfrak m_4$}; 
 \\
\node (g33)  {$\mathfrak m_2$}; &  &\node (l1)  {$\mathfrak l_1$}; &
   \\
 & \node (a0)  {$\mathfrak a_{0}$}; &  &
 \\
};
\draw[->, thin] (g34) -- (l1);  
\draw[->, thin] (g21g1) -- (l1);
\draw[->, thin] (g33) -- (a0);
\draw[->, thin] (l1) -- (a0);
\end{tikzpicture}
\\
{\footnotesize
\refstepcounter{pict}\label{Picture3DimLieChild}
Picture~\ref{Picture3DimLieChild}
}
\end{center}

\medskip
(iv) A complete list of non-isomorphic algebras $\mathfrak g\in\boldsymbol A$ with $\Theta(\mathfrak g)\in\phi_2(\mathcal A^s)$ consists of the associative commutative algebras $\mathfrak a_0$ and $\mathfrak c_i$ ($1\le i\le 11$) together with the following non-associative Jordan algebras (only non-zero commutation relations are listed):

$J_2$: $e_1e_1=e_1$, $e_2e_2=e_2$, $e_1e_3=e_3e_1=\frac12e_3$, $e_2e_3=e_3e_2=\frac12e_3$

$J_3$: $e_1e_1=e_1$, $e_2e_2=e_2$, $e_1e_3=e_3e_1=\frac12e_3$

$J_4$: $e_1e_1=e_1$, $e_1e_2=e_2e_1=\frac12e_2$, $e_1e_3=e_3e_1=e_3$

$J_5$: $e_1e_1=e_1$, $e_1e_2=e_2e_1=\frac12e_2$, $e_1e_3=e_3e_1=\frac12e_3$

$J_8$: $e_1e_1=e_1$, $e_1e_3=e_3e_1=\frac12e_3$.

\begin{center}
\begin{tikzpicture}
\matrix [column sep=7mm, row sep=5mm] {
 & & & \node (c11) { $\mathfrak c_{11}$}; & & \\
 & & \node (c9) {$\mathfrak c_{9}$}; & & \node (c10) {$\mathfrak c_{10}$}; & \\
\node (j2) { $J_2$}; & \node (c6) {$\mathfrak c_{6}$}; & \node (c7) {$\mathfrak c_{7}$}; & & \node (j3) {$ J_3$}; & \node (c8) {$\mathfrak c_{8}$}; \\
& & \node (j8)  {$J_8$}; & \node (j4) {$ J_4$}; & \node (c5) {$\mathfrak c_{5}$}; & \\
& \node (c2) {$\mathfrak c_{2}$}; & & \node (c3) {$\mathfrak c_{3}$}; & & \node (c4) {$\mathfrak c_{4}$};\\
& & & \node (c1) {$\mathfrak c_{1}$}; & & \node (j5) {$ J_5$}; \\
& & & \node (c0) {$\mathfrak a_{0}$}; & & \\
};
\draw[->, thin] (c11) -- (c9);
\draw[->, thin] (c11) -- (c10);
\draw[->, thin] (c9) -- (c6);
\draw[->, thin] (c9) -- (c7);
\draw[->, thin] (c9) -- (c8);
\draw[->, thin] (c10) -- (c7);
\draw[->, thin] (c10) -- (c8);
\draw[->, thin] (j2) -- (j8);
\draw[->, thin] (c6) -- (c5);
\draw[->, thin] (j3) -- (j8);
\draw[->, thin] (j3) -- (j4);
\draw[->, thin] (j3) -- (c4);
\draw[->, thin] (c7) -- (c5);
\draw[->, thin] (c8) -- (c5);
\draw[->,  thin] (j2) -- (c2);
\draw[->,  thin] (c6) -- (c2);
\draw[->, thin] (j8) -- (c3);
\draw[->,  thin] (c5) -- (c3);
\draw[->,  thin] (j4) -- (c3);
\draw[->,  thin] (c8) -- (c4);
\draw[->,  thin] (c2) -- (c1);
\draw[->,  thin] (c3) -- (c1);
\draw[->,  thin] (c4) -- (c1);
\draw[->,  thin] (c1) -- (c0);
\draw[->,  thin] (j5) -- (c0);
\end{tikzpicture}

{\footnotesize
\refstepcounter{pict}\label{Picture3DimJordChild}
Picture~\ref{Picture3DimJordChild}
}
\end{center}

In Picture~\ref{Picture3DimJordChild}  we include all possible degenerations inside the set $\phi_2(\mathcal A^s)\,(\subseteq \mathcal J)$ as these were determined in~\cite{KashubaShestakov2007,GorshkovKaygorodovPopov2021}.
In particular, the results of those papers show that $\phi_2(\mathcal A^s)$ is an algebraic subset of $\mathcal J$.

\medskip
(v) In Table~\ref{TableDimmAnn} below we collect information on certain algebra invariants, computed for some of the algebras listed in Table~\ref{Table3dimAssocClass}.

{\begin{center}

\refstepcounter{tbn}\label{TableDimmAnn}
\footnotesize
\begin{tabular}{c|c|c|c|c|c|c|c|c|c|c|c} %
$\mathfrak a$               & $\mathfrak a_0$  & $\mathfrak a_2$  & $\mathfrak a_3(\kappa)$  & $\mathfrak l_1$ &  $\mathfrak c_1$ &  $\mathfrak c_3$  & $\mathfrak a_7$  & $\mathfrak a_8$  & $\mathfrak a_9$  & $\mathfrak a_{10}$ & $\mathfrak a_{11}$\\
\hline
$\dim{\rm ann}_L\mathfrak a$ & 3 & 2 & 1&1 &2 & 1 & 0&1 &2 &1 &1\\
\hline
$\dim{\rm ann}_R\mathfrak a$ & 3 & 2 & 1&1 &2 & 1 & 1 & 0 &1 & 2 &0\\
\hline
$\dim{\rm Der}$ & 9 & 4 &4 &6 &5 &4 &3 &3 &3 &3 &2\\
\hline
$\dim{\rm Der}_{(1,0,1)}$ & 9  & 5 & 3  & 3  &5  &3  &2  &5  &5  &3 &5 \\
\end{tabular}
\\[2ex]
Table~\ref{TableDimmAnn} (with $\kappa\in\mathbb K$).\\[2ex]

\end{center}}

{\bf Notation.} At this point it will be convenient to introduce some more notation.

(i) We will say that algebras $\mathfrak g_1$, $\mathfrak g_2\in A$ form a $\{\mathfrak g,\tilde{\mathfrak g}\}$-pair if $\Theta(\tilde{\mathfrak g_2})\in O(\Theta({\mathfrak g_1}))$ but $\Theta(\mathfrak g_2)\not\in O(\Theta({\mathfrak g_1}))$.
In such a case we also have that $\Theta(\tilde{\mathfrak g_1})\in O(\Theta({\mathfrak g_2}))$ and $\Theta(\mathfrak g_1)\not\in O(\Theta({\mathfrak g_2}))$.
[Note that for $\boldsymbol\lambda,\boldsymbol\mu\in\boldsymbol\Lambda$ we have that $\phi(\boldsymbol\mu)\in O(\boldsymbol\lambda)$ if, and only if, $\boldsymbol\lambda\in O(\phi(\boldsymbol\mu))\,(=\phi(O(\boldsymbol\mu)))$, and this last statement holds if, and only if, $\phi(\boldsymbol\lambda)\in O(\boldsymbol\mu)$ since $\phi^2={\rm id}_{\boldsymbol\Lambda}$.]

(ii) For the sake of simplicity, in various occasions in the rest of the paper we will use, for $\mathfrak g\in \boldsymbol A$, symbol $\mathfrak g^o$ to mean $O(\Theta(\mathfrak g))$ --- in particular $\mathfrak g^o\subseteq\boldsymbol\Lambda$.
In this notation we have $\overline{\mathfrak a_0^o}=\mathfrak a_0^o=\{{\bf0}\}\subseteq\boldsymbol\Lambda$.

\medskip
In the following remark we collect some observations regarding the algebras in Table~\ref{Table3dimAssocClass}.

\begin{remark}\label{Remark5.1}
(i) The following is a complete list of pairs of algebras from Table~\ref{Table3dimAssocClass} that form $\{\mathfrak g,\tilde{\mathfrak g}\}$-pairs:
$
\{\mathfrak a_4,\mathfrak a_5\}$,\quad $\{\mathfrak a_7,\mathfrak a_8\}$,\quad $\{\mathfrak a_9,\mathfrak a_{10}\}$,\quad $\{\mathfrak a_{11},\mathfrak a_{12}\}.$

\medskip
(ii)
We have the following intersections of $\mathcal A^s$ with some of our familiar algebraic sets.
\begin{gather*}
\mathcal A^s\cap\mathcal K=\{{\bf0}\}\cup\mathfrak l_1^o\\
\mathcal A^s\cap\mathcal M^{**}=\{{\bf0}\}\cup\mathfrak a_4^o\cup\mathfrak a_5^o\cup\mathfrak a_6^o\cup\mathfrak l_1^o\\
\mathcal A^s\cap\mathcal M^*=\{{\bf0}\}\cup\mathfrak a_4^o\cup\mathfrak a_5^o\\
\mathcal A^s\cap\mathcal C=\{{\bf0}\}\cup(\bigcup\limits_{1\le i\le11}\mathfrak c_i^o)\\
\mathcal A^s\cap\mathcal T=\{{\bf0}\}\cup\mathfrak l_1^o\cup\mathfrak c_1^o\cup\mathfrak c_3^o\cup\mathfrak c_5^o\cup\mathfrak a_2^o\cup(\bigcup\limits_{\kappa\in\mathbb K}\mathfrak a_3(\kappa)^o)\\
\mathcal A^s\cap(\mathcal C\cap \mathcal T)=\{{\bf0}\}\cup\mathfrak c_1^o\cup\mathfrak c_3^o\cup\mathfrak c_5^o\\
\mathcal A^s\cap\mathcal B=\{{\bf0}\}\cup\mathfrak l_1^o\cup\mathfrak c_1^o\cup\mathfrak c_3^o\cup\mathfrak a_2^o\cup(\bigcup\limits_{\kappa\in\mathbb K}\mathfrak a_3(\kappa)^o)
\end{gather*}
By setting $\tilde{\mathcal T}=\phi(\mathcal T)\,(=\{\boldsymbol\lambda=(\lambda_{ijk})\in\boldsymbol\Lambda\colon \sum_{j=1}^n\lambda_{jij}=0$ for all $1\le i\le n\})$,
we see that $\mathcal A^s\cap\mathcal T=\mathcal A^s\cap\tilde{\mathcal T}$ and that $\mathcal A^s\cap\mathcal B\subseteq\mathcal A^s\cap(\mathcal T\cap\tilde{\mathcal T})$.

\medskip
(iii)
(a) $\Theta(\mathfrak l_1)\in O(\boldsymbol\eta)$ and $\Theta(\mathfrak c_1)\in O(\boldsymbol\delta)$ with $\boldsymbol\eta$ and $\boldsymbol\delta$ as in~\cite{IvanovaPallikaros2019}.
Also recall  from~\cite{IvanovaPallikaros2019} that $\overline{O(\boldsymbol\eta)}=\{{\bf0}\}\cup O(\boldsymbol\eta)$ and $\overline{O(\boldsymbol\delta)}=\{{\bf0}\}\cup O(\boldsymbol\delta)$.

(b) If $\boldsymbol\lambda\in\mathcal M^{**}-\mathcal M^*$, then $\boldsymbol\eta\in \overline{O(\boldsymbol\lambda)}$, see \cite[Lemma 4.4]{IvanovaPallikaros2019}.

(c) If $\boldsymbol\lambda\in\boldsymbol\Lambda-\mathcal M^{**}$, then $\boldsymbol\delta\in \overline{O(\boldsymbol\lambda)}$, see \cite[Lemma 5.4]{IvanovaPallikaros2019}.

\medskip
(iv) Clearly $\mathcal C\cap \mathcal A^s=\{\boldsymbol\lambda\in\mathcal A^s\colon \boldsymbol\lambda=\tilde{\boldsymbol\lambda}\}$ so for $\boldsymbol\mu\in\mathcal C\cap \mathcal A^s$ we have that $\boldsymbol\mu=\phi_2(\boldsymbol\mu)\in\mathcal J$.
Invoking the fact that $\mathcal J\subseteq\mathcal C$ we get  $\mathcal C\cap \mathcal A^s=\mathcal J\cap \mathcal A^s$.
It follows that the degeneration picture inside the algebraic subset $\mathcal C\cap \mathcal A^s$ of $\mathcal A^s$ coincides with the degeneration picture inside the algebraic subset $\mathcal J\cap\mathcal A^s$ of $\mathcal J$.

\end{remark}

\section{Degenerations of 3-dimensional complex associative algebras}\label{SectionDegen3dim}

We continue with our hypothesis that $\F=\mathbb C$ and $n=3$.
Also recall that $\mathbb K=\{x+yi\colon x,y\in\mathbb R$ with $x>0$ or ($x=0$ and $y>0)\}\subseteq\mathbb C$.

\medskip
Our aim in this section is to determine the complete degeneration picture inside $\mathcal A^s$.
Our general plan in order to achieve this is to make use of the observations in Remark~\ref{Remark5.1} together with sets of the form $T=\phi_1^{-1}(T_1)\cap\phi_2^{-1}(T_2)\cap \mathcal A^s$, where the sets $T_1$, $T_2$ ($T_1\subseteq \mathcal L$, $T_2\subseteq\mathcal J$) are algebraic sets which are also unions of $G$-orbits under the action of $G$ on $\boldsymbol\Lambda$ we are considering.
Since the maps $\phi_1$ and $\phi_2$ are continuous in the Zariski topology, we can see that such sets $T$ are also  algebraic and, moreover, Remark~\ref{RemarkObservations}(ii) ensures that they consist of a union of $G$-orbits.
(Note that the intersection of subsets of $\boldsymbol\Lambda$ which are unions of $G$-orbits, if it is non-empty, is also a union of $G$-orbits.)
In order to obtain the complete degeneration picture inside $\mathcal A^s$, we will need to employ additional arguments, depending for example on necessary conditions for degeneration or arguments using the ideas involved in the preliminary lemmas in Section~\ref{SectionPrelim} of this paper.

The process will be completed through a sequence of steps.
The corresponding sets $T_1$ and $T_2$ at each step of the process are chosen by `filtering out' closed sets as we go from bottom to top in Pictures~\ref{Picture3DimLieChild} and~\ref{Picture3DimJordChild} respectively.

\bigskip
{\bf Step 1.} We consider the set $S_1=\phi_1^{-1}(\{{\bf0}\}\cup\mathfrak m_2^o)\cap\phi_2^{-1}(\{{\bf0}\}\cup J_5^o)\cap\mathcal A^s=\{{\bf0}\}\cup \mathfrak a_4^o\cup\mathfrak a_5^o=\mathcal A^s\cap\mathcal M^*$, a closed set which is also a union of orbits from the above discussion.
From Remark~\ref{RemarkObservations}(v) and Remark~\ref{Remark5.1}(i),(ii) we get that both the sets $\{{\bf0}\}\cup\mathfrak a_4^o$ and $\{{\bf0}\}\cup\mathfrak a_5^o$ are closed.
(Alternatively, we could use the results in~\cite{IvanovaPallikaros2019}.)

\bigskip
{\bf Step 2.} We consider the set $S_2=\phi_1^{-1}(\{{\bf0}\}\cup \mathfrak l_1^o\cup\mathfrak m_3^o)\cap\phi_2^{-1}(\{{\bf0}\}\cup J_5^o)\cap\mathcal A^s=\{{\bf0}\}\cup \mathfrak l_1^o\cup \mathfrak a_6^o\subseteq \mathcal A^s\cap\mathcal M^{**}$.
From Remark~\ref{Remark5.1}(ii),(iii) we get that $\Theta(\mathfrak l_1)\in\overline{\mathfrak a_6^o}$ since $\Theta(\mathfrak a_6)\in\mathcal M^{**}-\mathcal M^*$.
Note that $\mathcal A^s\cap\mathcal M^{**}=(\{{\bf0}\}\cup\mathfrak a_4^o\cup\mathfrak a_5^o)\cup(\{{\bf0}\}\cup\mathfrak l_1^o\cup\mathfrak a_6^o)$, a union of two closed sets.
Moreover, the set $\{{\bf0}\}\cup\mathfrak l_1^o$ is closed, again by Remark~\ref{Remark5.1}(iii).

\bigskip
{\bf Step 3.} We consider the set $S_3=\phi_1^{-1}(\{{\bf0}\}\cup \mathfrak l_1^o)\cap\phi_2^{-1}(\{{\bf0}\}\cup \mathfrak c_1^o \cup \mathfrak c_3^o)\cap\mathcal A^s = \{{\bf0}\}\cup \mathfrak l_1^o\cup \mathfrak c_1^o\cup \mathfrak c_3^o\cup \mathfrak a_2^o\cup (\bigcup\limits_{\kappa\in\mathbb K}\mathfrak a_3(\kappa)^o)=\mathcal A^s\cap\mathcal B\subseteq\mathcal A^s\cap\mathcal T$.
The sets $\{{\bf0}\}\cup\mathfrak l_1^o\,(=\mathcal A^s\cap\mathcal K)$ and $\{{\bf0}\}\cup\mathfrak c_1^o\cup \mathfrak c_3^o\,(=\mathcal A^s\cap\mathcal B\cap\mathcal C\cap\mathcal T)$ are both closed, see Remark~\ref{Remark5.1}(ii).
Hence $\{{\bf0}\}\cup\mathfrak l_1^o\cup \mathfrak c_1^o\cup \mathfrak c_3^o$ is a closed subset of $\mathcal A^s$.
Moreover, from Remark~\ref{Remark5.1}(ii),(iii) we get that $\mathfrak c_3$, $\mathfrak a_2$ and $\mathfrak a_3(\kappa)$, for all $\kappa\in\mathbb K$, all degenerate to $\mathfrak c_1$ since $\Theta(\mathfrak c_1)\in O(\boldsymbol\delta)$.

From the information about $\dim{\rm ann}_L$ in Table~\ref{TableDimmAnn}, we also get that $\Theta(\mathfrak a)\not\in\overline{\mathfrak a_2^o}$ for all $\mathfrak a\in\{\mathfrak l_1,\mathfrak c_3\}\cup\{\mathfrak a_3(\kappa)\colon \kappa\in\mathbb K\}$.
In particular, $\overline{\mathfrak a_2^o}=\{{\bf0}\}\cup\mathfrak c_1^o\cup \mathfrak a_2^o $.

Also from Table~\ref{TableDimmAnn}, looking now at $\dim{\rm Der}$ of the corresponding algebras, we see that $\mathfrak a_3(\kappa)\not\to\mathfrak a_2$ and $\mathfrak a_3(\kappa)\not\to\mathfrak c_3$ for all $\kappa\in\mathbb K$ and, moreover, there is no degeneration between any two members of the infinite family $\mathfrak a_3(\kappa)$ with $\kappa\in\mathbb K$.

Next, we investigate whether there is a degeneration from $\mathfrak a_3(\kappa)$ to $\mathfrak l_1$, at least for some $\kappa\in\mathbb K$, and for this it will be convenient to consider the cases $\kappa=2$ and $\kappa\in\mathbb K-\{2\}$ separately.

(a) We assume $\kappa=2$:
Let $\boldsymbol\lambda={\bf221}+{\bf331}+2({\bf321})\in\boldsymbol\Lambda$.
Then $\boldsymbol\lambda=\Theta(\mathfrak a_3(\kappa))$,
so in Example~\ref{ExamplApplA} we established in fact that $\mathfrak a_3(\kappa)\to\mathfrak l_1$ in the case $\kappa=2$.

(b) We assume $\kappa\in\mathbb K-\{2\}$:
Let $\boldsymbol\nu=\Theta(\mathfrak a_3(\kappa))={\bf221}+{\bf331}+\kappa({\bf321})\in\boldsymbol\Lambda$.
Also let $\alpha\in\mathbb C$ be a root of the polynomial $x^2+\kappa x+1\in\mathbb C[x]$.
Comparing with Lemma~\ref{Lemma5.2}, we see that our hypothesis on $\kappa$ ensures that $\alpha^2\ne1$ and that $\boldsymbol\lambda={\bf231}+-\alpha^2({\bf321})\in O(\boldsymbol\nu)$.
Invoking now Example~\ref{ExamplApplB} with $\beta=-\alpha^2\,(\ne-1)$ we get that $\overline{O(\boldsymbol\nu)}\cap\mathcal K=\overline{O(\boldsymbol\lambda)}\cap\mathcal K=\{{\bf0}\}$.
We conclude that $\mathfrak a_3(\kappa)\not\to\mathfrak l_1$ whenever $\kappa\in\mathbb K-\{2\}$.
[To see this it might be useful to recall that $\overline{O(\boldsymbol\nu)}\subseteq\mathcal A^s$ since $\boldsymbol\nu\in\mathcal A^s$, and also that $\mathcal K\cap\mathcal A^s=\{{\bf0}\}\cup\mathfrak l_1^o$.]

This completes the degeneration picture inside the closed set $S_3$:

\begin{center}
\begin{tikzpicture}
\matrix [column sep=7mm, row sep=5mm] {
\node (a6)  {$\mathfrak a_{6}$}; &   \node (a3k2)  {$\mathfrak a_{3}(\kappa=2)$}; &   \node (a3kNot2)  {$\mathfrak a_{3}(\kappa\ne2)$}; & &\node (c3)  {$\mathfrak c_{3}$}; & \node (a2)  {$\mathfrak a_{2}$};
 \\
& \node (a4)  {$\mathfrak a_4$}; & \node (l1)  {$\mathfrak l_{1}$};  & \node (c1)  {$\mathfrak c_{1}$}; &\node (a5)  {$\mathfrak a_5$};
   \\
& & \node (a0)  {$\mathfrak a_{0}$}; &  &
 \\
};
\draw[->, thin] (a3kNot2) -- (c1);
\draw[->, thin] (a3k2) -- (c1);
\draw[->, thin] (a3k2) -- (l1);
\draw[->, thin] (a6) -- (l1); 
\draw[->, thin] (c3) -- (c1);
\draw[->, thin] (a2) -- (c1);
\draw[->, thin] (l1) -- (a0);
\draw[->, thin] (a4) -- (a0);
\draw[->, thin] (c1) -- (a0);
\draw[->, thin] (a5) -- (a0);
\end{tikzpicture}

{\footnotesize
\refstepcounter{pict}\label{Picture3DimDegenAssocPart1}
Picture~\ref{Picture3DimDegenAssocPart1} (with $\kappa\in\mathbb K$)
}
\end{center}

\bigskip
{\bf Step 4.} We consider the set $S_4=\phi_1^{-1}(\{{\bf0}\}\cup \mathfrak l_1^o\cup \mathfrak m_4^o)\cap\phi_2^{-1}(\{ J_8^o\cup \mathfrak c_3^o\cup\mathfrak c_1^o\cup\{{\bf0}\}\})\cap \mathcal A^s=\mathfrak a_9^o\cup\mathfrak a_{10}^o\cup\mathfrak a_2^o\cup\mathfrak c_3^o\cup\mathfrak c_1^o\cup\mathfrak l_1^o\cup\{{\bf0}\}\cup(\bigcup\limits_{\kappa\in\mathbb K}\mathfrak a_3(\kappa)^o)$. 

We can make the following observations.

\begin{enumerate}[(i)]
\item
We know from Remark~\ref{Remark5.1}(i) that $\{\mathfrak a_9,\mathfrak a_{10}\}$ is a $\{\mathfrak g,\tilde{\mathfrak g}\}$-pair, so $\mathfrak a_9\not\to\mathfrak a_{10}$ and $\mathfrak a_{10} \not\to\mathfrak a_{9}$ in view of Remark~\ref{RemarkObservations}(v).
(This can also be seen by either considering $\dim{\rm ann}_L$ and $\dim{\rm ann}_R$ or $\dim{\rm Der}$ of these algebras --- see Table~\ref{TableDimmAnn}).

\item
From Step 3, the subset $S_3=\mathfrak a_2^o\cup\mathfrak c_3^o\cup\mathfrak c_1^o\cup\mathfrak l_1^o\cup\{{\bf0}\}\cup(\bigcup\limits_{\kappa\in\mathbb K}\mathfrak a_3(\kappa)^o)$ of $S_4$ is closed and we know everything about the degenerations between the members of $S_3$.
Hence, in order to determine all possible degenerations between the members of the closed set $S_4$ it suffices to determine all possible degenerations from $\mathfrak a_9$ (or $\mathfrak a_{10}$) to each of the members of $S_3$.

\item
Let $\boldsymbol\lambda,\boldsymbol\mu\in\boldsymbol\Lambda$ with $\boldsymbol\lambda={\bf322}+{\bf333}$ and $\boldsymbol\mu={\bf232}+{\bf333}\,(=\tilde{\boldsymbol\lambda})$.
Then $\boldsymbol\lambda=\Theta(\mathfrak a_9)$ and $\boldsymbol\mu=\Theta(\mathfrak a_{10})$, see Table~\ref{Table3dimAssocClass}.

For $t\in\mathbb C-\{0\}$, define matrices $b_1(t),b_2(t)\in\mathop{{\rm GL}(3,\mathbb C)}$ as follows
\[
b_1(t)=
\begin{pmatrix}
t&0&-1\\
0&0&1\\
0&t&0
\end{pmatrix}
\quad\mbox{and}\quad
b_2(t)=
\begin{pmatrix}
t&-1&0\\
0&1&0\\
0&0&t
\end{pmatrix}.
\]
Then $\boldsymbol\lambda\, b_1(t)={\bf231}+t({\bf233})+t({\bf222})$.
Invoking Lemma~\ref{LemmaA} (see also argument in Example~\ref{ExamplApplA}), we get that $\mathfrak a_9\to\mathfrak a_2=\Theta^{-1}({\bf231})$.
Similarly, by computing $\boldsymbol\mu\, b_2(t)$ we also observe that $\mathfrak a_{10}\to\mathfrak a_2$.
These are examples of construction of degeneration via Lemma~\ref{LemmaC}  ---
observe that the subspace $\mathbb C$-span$(-e_1+e_2)$ is a subalgebra of both $\mathfrak a_9$ and $\mathfrak a_{10}$. 
[Remark: The observation that the pair of algebras $\{\mathfrak a_9,\mathfrak a_{10}\}$ is a  $\{\mathfrak g,\tilde{\mathfrak g}\}$-pair together with the fact that ${\Theta(\tilde{\mathfrak a_2})}\in \mathfrak a_2^o$ ensure that $\mathfrak a_9\to\mathfrak a_2$ if, and only if,  $\mathfrak a_{10}\to\mathfrak a_2$.]

\item
From the transitivity of degenerations we also get (the facts already known from~\cite{IvanovaPallikaros2019}, see Remark~\ref{Remark5.1}(iii)) that $\mathfrak g\to \mathfrak c_1$, $\mathfrak g\to\mathfrak a_0$ for $\mathfrak g\in\{\mathfrak a_9,\mathfrak a_{10}\}$, since $\mathfrak g\to \mathfrak a_2$ and $\mathfrak a_2\to \mathfrak c_1$, $\mathfrak a_2\to\mathfrak a_0$.

\item
Finally, we use the information from Table~\ref{TableDimmAnn} on $\dim{\rm ann}_L$ and  $\dim{\rm ann}_R:$

$\dim{\rm ann}_L$ gives$:$ $\mathfrak a_9\not\to \mathfrak l_1$, $\mathfrak a_9\not\to \mathfrak c_3$,   $\mathfrak a_9\not\to \mathfrak a_3(\kappa)$, for all $\kappa\in\mathbb K$,  while

$\dim{\rm ann}_R$ gives$:$ $\mathfrak a_{10}\not\to \mathfrak l_1$, $\mathfrak a_{10}\not\to \mathfrak c_3$,  $\mathfrak a_{10}\not\to \mathfrak a_3(\kappa)$, for all $\kappa\in\mathbb K$.

[Alternatively, we could use an obvious modification of the remark immediately above in order to obtain the corresponding information for $\mathfrak a_{10}$ given the information on $\mathfrak a_9$.]
\end{enumerate}

The observations (i)--(v) above supply sufficient information in order to complete the degeneration picture inside set $S_4$.

\bigskip
{\bf Step 5.} We consider  the set $S_5=\phi_1^{-1}(\{{\bf0}\}\cup \mathfrak l_1^o\cup\mathfrak m_4^o)\cap\phi_2^{-1}(J_4^o\cup\mathfrak c_3^o\cup\mathfrak c_1^o\cup\{{\bf0}\}) \cap \mathcal A^s= \mathfrak a_7^o\cup \mathfrak a_8^o\cup\mathfrak a_2^o\cup\mathfrak l_1^o\cup\mathfrak c_3^o\cup\mathfrak c_1^o\cup\{{\bf0}\}\cup (\bigcup\limits_{\kappa\in\mathbb K}\mathfrak a_3(\kappa)^o) = \mathfrak a_7^o\cup \mathfrak a_8^o \cup S_3$.
Now $\{\mathfrak a_7, \mathfrak a_8\}$ is a   $\{\mathfrak g,\tilde{\mathfrak g}\}$-pair and the set $S_3$ is closed.
Comparing with the discussion in Step~4, we get that $\mathfrak a_7\not\to\mathfrak a_8$, $\mathfrak a_8\not\to\mathfrak a_7$ and, moreover, in order to complete the degeneration picture inside the closed set $S_5$ it suffices to determine whether any of the members of the set $S_3=\mathcal A^s\cap\mathcal B$ belong to $\overline{\mathfrak g^o}$, for $\mathfrak g\in\{\mathfrak a_7, \mathfrak a_8\}$.
We aim to use Lemma~\ref{LemmaB}.

Let $\boldsymbol\lambda'={\bf121}+{\bf211}+{\bf222}+{\bf323}\in\boldsymbol\Lambda$, so $\boldsymbol\lambda'=\Theta(\mathfrak a_7)$ from Table~\ref{Table3dimAssocClass}.
Also let $B$ be the Borel subgroup of all upper triangular matrices in $\GlC$.
It will be convenient to consider the basis $(e_1',e_2',e_3')$ of the underlying vector space $V$, where $e_1'=e_1$, $e_2'=e_3$, $e_3'=e_2$.
Note that the subspace $\mathbb C$-span$(e_1',e_2')$ is a subalgebra of $\mathfrak a_7$ isomorphic to the 2-dimensional Abelian algebra so the commutation relations remain very simple (many of the coefficients remain zero) when we act by the subgroup $B$.
The structure vector of $\mathfrak a_7$ relative to the basis $(e_1',e_2',e_3')$ is $\boldsymbol\lambda={\bf131}+{\bf311}+{\bf333}+{\bf232}$.
For $b=(b_{ij})\in B$, in particular $b_{ij}=0$ if $i>j$ and $b_{11}b_{22}b_{33}\ne0$, we then have $\boldsymbol\lambda b=b_{33}{\bf131} +b_{33}{\bf311} + b_{33}{\bf232} + \big(\frac{b_{12}b_{33}}{b_{11}}\big){\bf321}+\big(\frac{b_{13}b_{33}}{b_{11}}\big){\bf331}+b_{33}{\bf333}\in\boldsymbol\lambda'G$.

Clearly, from the above expression for $\boldsymbol\lambda b$, the following polynomials  all belong to ${\bf I}(\boldsymbol\lambda B)$, the vanishing ideal of $\boldsymbol\lambda B$:
\begin{gather*}
X_{11i},\ X_{12i},\ X_{21i},\ X_{22i}, \quad(1\le i\le3)\\
X_{231},\ X_{233},\ X_{322},\ X_{323},\ X_{332},\ X_{132},\ X_{133},\ X_{312},\ X_{313},\\
X_{232}-X_{333},\ X_{131}-X_{333},\ X_{311}-X_{333}.
\end{gather*}

Now let $\boldsymbol\mu=(\mu_{ijk})\in\overline{\boldsymbol\lambda B}$.
Then $\ev_{\boldsymbol\mu}(f)=0$ for all $f\in{\bf I}(\boldsymbol\lambda B)$, hence
\begin{gather*}
\mu_{11i}=0,\ \mu_{12i}=0,\ \mu_{21i=0},\ \mu_{22i}=0, \quad(1\le i\le3)\\
\mu_{231}=0,\ \mu_{233}=0, \mu_{322}=0,\ \mu_{323}=0,\ \mu_{332}=0,\ \mu_{132}=0,\ \mu_{133}=0,\ \mu_{312}=0,\ \mu_{313}=0,\\
\mu_{232}=\mu_{333}=\mu_{131}=\mu_{311}.
\end{gather*}
It follows that the only coefficients $\mu_{ijk}$ which can possibly be non-zero are: $\mu_{131}$, $\mu_{232}$, $\mu_{311}$, $\mu_{321}$, $\mu_{331}$ and $\mu_{333}$, with the additional restriction that $\mu_{232}=\mu_{333}=\mu_{131}=\mu_{311}$.

At this point we make the further assumption that $\boldsymbol\mu\in\mathcal B$, so we have that $\boldsymbol\mu=(\mu_{ijk})\in\overline{\boldsymbol\lambda B}\cap \mathcal B$.
From the defining conditions of $\mathcal B$, see Section~\ref{SectionPhi1Phi2}, we must have that $\mu_{331}\mu_{133}+\mu_{332}\mu_{233}+\mu_{333}\mu_{333}=0$.
But $\mu_{133}=0=\mu_{233}$ from our previous considerations, so $\mu_{333}^2=0$ and hence $\mu_{333}=0$.
We conclude that all coefficients $\mu_{ijk}$ are equal to zero except possibly $\mu_{321}$ and $\mu_{331}$.
Next, we impose the further restriction that $\boldsymbol\mu\ne{\bf0}$.
Hence it suffices to consider the following cases$:$ (i) $\mu_{321}=0$, $\mu_{331}\ne0$ and (ii) $\mu_{321}\ne0$.
[Note that there is no guarantee that any of cases (i), (ii) described above actually occur, as some other polynomials in ${\bf I}(\boldsymbol\lambda B)$ or some other defining conditions for $\mathcal B$ not already listed above could possibly impose even further restrictions on the coefficients $\mu_{ijk}$.
However, our goal at the moment is, by making use of Lemma~\ref{LemmaB}, to exclude the possibility of degeneration from $\mathfrak a_7$ to various members of the set $\Theta^{-1}(\mathcal A^s\cap\mathcal B) =\Theta^{-1}(S_3)$.]

Case (i): $\mu_{231}=0$, $\mu_{331}\ne0$.
Clearly, $\Theta(\mathfrak c_1)\in O(\boldsymbol\mu)$, so $\Theta^{-1}(\boldsymbol\mu)\simeq\mathfrak c_1$.

Case (ii): $\mu_{321}\ne0$. Then $\Theta^{-1}(\boldsymbol\mu)\simeq \mathfrak g$ where $\Theta(\mathfrak g)={\bf321}+\gamma({\bf331})$ for some $\gamma\in\mathbb C$.
We will consider subcases (ii)(a) and (ii)(b) given below$:$

Subcase (ii)(a): $\gamma=0$: Then clearly $\Theta^{-1}(\boldsymbol\mu)\simeq \mathfrak a_2$.

Subcase (ii)(b): $\gamma\ne0$: We consider $({\bf321}+\gamma{\bf331})g$, with
$g=${\footnotesize$\begin{pmatrix}
1&0&0\\
0&1&1\\
0&0&-\gamma^{-1}
\end{pmatrix}$}$\in\GlC$.
It is easy to compute that $({\bf321}+\gamma{\bf331})g=-\frac1\gamma({\bf321})$, so again $\Theta^{-1}(\boldsymbol\mu)\simeq \mathfrak a_2$.

Summing up, we have shown up to this point that $\overline{\boldsymbol\lambda B}\cap\mathcal B\subseteq\{{\bf0}\}\cup \mathfrak c_1^o\cup\mathfrak a_2^o$.
Let $U=\mathfrak l_1^o\cup\mathfrak c_3^o\cup (\bigcup\limits_{\kappa\in\mathbb K}\mathfrak a_3(\kappa)^o)$.
Then $U\subseteq \mathcal B$ from Remark~\ref{Remark5.1}(ii), hence $\mathcal B\cap U=U$ and $\overline{\boldsymbol\lambda B}\cap U=(\overline{\boldsymbol\lambda B}\cap\mathcal B)\cap U=\varnothing$.
It follows from Remark~\ref{RemarkBorelSubgroup}(iv)that $\overline{O(\boldsymbol\lambda)}\cap U=\varnothing$.
We conclude that  $\mathfrak a_7\not\to\mathfrak l_1$, $\mathfrak a_7\not\to\mathfrak c_3$ and $\mathfrak a_7\not\to\mathfrak a_3(\kappa)$, for all $\kappa\in\mathbb K$.

Recalling the expression for $\boldsymbol\lambda b$ obtained above and invoking  Lemma~\ref{LemmaA} with $t\in\mathbb C-\{0\}$ (and an argument as in Example~\ref{ExamplApplA}) we also get$:$

(i) $\mathfrak a_7\to\mathfrak a_2$, by setting $b_{33}=t$, $b_{12}=t$, $b_{11}=t^2$ and $b_{13}=t^2$, and

(ii) $\mathfrak a_7\to\mathfrak c_1$, by setting $b_{33}=t$, $b_{12}=t$, $b_{11}=t$ and $b_{13}=1$.

Alternatively, we can observe that in Example~\ref{ExamplApplC} a degeneration $\mathfrak a_7\to\mathfrak a_2$ was in fact constructed since $\Theta(\mathfrak a_7)={\bf121}+{\bf211}+{\bf222}+{\bf323}$ and ${\bf213}\in\mathfrak a_2^o$.
We can then invoke the transitivity of degenerations to establish that $\mathfrak a_7\to\mathfrak c_1$ and $\mathfrak a_7\to\mathfrak a_0$.
(These last two facts are already known since $\Theta(\mathfrak a_7)\in\boldsymbol\Lambda-\mathcal M^{**}$, see Remark~\ref{Remark5.1}(ii),(iii).)
This completes the degeneration picture inside the closed set $S_5$ since our results on degenerations still hold if we replace $\mathfrak a_7$ by $\mathfrak a_8$ in view of the fact that $\boldsymbol\lambda\in O(\tilde{\boldsymbol\lambda})$ for all $\boldsymbol\lambda\in S_3$.

\begin{center}
\begin{tikzpicture}
\matrix [column sep=7mm, row sep=5mm] {
 &    &    &   \node (a9)  {$\mathfrak a_{9}$};  &   \node (a7)  {$\mathfrak a_{7}$};  &   \node (a8)  {$\mathfrak a_{8}$};  &  \node (a10)  {$\mathfrak a_{10}$};
\\
\node (a6)  {$\mathfrak a_{6}$}; & \node (a3k2)  {$\mathfrak a_{3}(\kappa=2)$}; &   \node (a3kNot2)  {$\mathfrak a_{3}(\kappa\ne2)$}; & &\node (c3)  {$\mathfrak c_{3}$}; & \node (a2)  {$\mathfrak a_{2}$}; & &
 \\
& \node (a4)  {$\mathfrak a_4$}; & \node (l1)  {$\mathfrak l_{1}$};  & \node (c1)  {$\mathfrak c_{1}$}; &\node (a5)  {$\mathfrak a_5$}; & &
   \\
& & \node (a0)  {$\mathfrak a_{0}$}; &  & & &
 \\
};
\draw[->, thin] (a10) -- (a2);
\draw[->, thin] (a8) -- (a2);
\draw[->, thin] (a7) -- (a2);
\draw[->, thin] (a9) -- (a2);
\draw[->, thin] (a3kNot2) -- (c1);
\draw[->, thin] (a3k2) -- (c1);
\draw[->, thin] (a3k2) -- (l1);
\draw[->, thin] (a6) -- (l1); 
\draw[->, thin] (c3) -- (c1);
\draw[->, thin] (a2) -- (c1);
\draw[->, thin] (l1) -- (a0);
\draw[->, thin] (a4) -- (a0);
\draw[->, thin] (c1) -- (a0);
\draw[->, thin] (a5) -- (a0);
\end{tikzpicture}

{\footnotesize
\refstepcounter{pict}\label{Picture3DimDegenAssocPart2}
Picture~\ref{Picture3DimDegenAssocPart2} (with $\kappa\in\mathbb K$)
}
\end{center}

\bigskip
{\bf Step 6.} We consider the set $S_6=\phi_1^{-1}(\{{\bf0}\}\cup \mathfrak l_1^o\cup\mathfrak m_4^o)\cap\phi_2^{-1}(J_3^o\cup J_4^o\cup J_8^o\cup \mathfrak c_4^o\cup \mathfrak c_3^o\cup \mathfrak c_1^o\cup\{{\bf0}\} )\cap\mathcal A^s=\mathfrak a_{11}^o\cup \mathfrak a_{12}^o\cup \mathfrak a_9^o\cup \mathfrak a_{10}^o\cup \mathfrak a_7^o\cup \mathfrak a_8^o\cup \mathfrak c_4^o\cup \mathfrak a_2^o\cup \mathfrak c_3^o\cup \mathfrak c_1^o\cup \mathfrak l_1^o\cup \{{\bf0}\}\cup(\bigcup\limits_{\kappa\in\mathbb K}\mathfrak a_3(\kappa)^o)=\mathfrak a_{11}^o\cup\mathfrak a_{12}^o\cup\mathfrak c_4^o\cup S_4\cup S_5$. 

The set $S_4\cup S_5$ is closed as it is the union of two closed sets and we know everything about the degeneration picture in this set from the previous steps of this process.
We also know from the information given in Picture~\ref{Picture3DimJordChild} that $\overline{\mathfrak c_4^o}=\{{\bf0}\}\cup \mathfrak c_1^o\cup\mathfrak c_4^o$.

Next, the algebras $\mathfrak a_{11}$ and $\mathfrak a_{12}$ form a  $\{\mathfrak g,\tilde{\mathfrak g}\}$-pair so there cannot be any degeneration between them.
Moreover, invoking the facts that the pairs of algebras $\{\mathfrak a_7,\mathfrak a_8\}$ and $\{\mathfrak a_9,\mathfrak a_{10}\}$  also form $\{\mathfrak g,\tilde{\mathfrak g}\}$-pairs  while ${\Theta(\tilde{\mathfrak a})}\in \mathfrak a^o$ for all $\mathfrak a\in\boldsymbol A$ with $\Theta(\mathfrak a)\in S_6-( \mathfrak a_7^o\cup\mathfrak a_8^o\cup\mathfrak a_9^o\cup\mathfrak a_{10}^o\cup\mathfrak a_{11}^o\cup\mathfrak a_{12}^o)$, the set of algebras to which $\mathfrak a_{12}$ degenerates can easily be obtained once we know the set of algebras to which $\mathfrak a_{11}$ degenerates (compare with Remark~\ref{RemarkObservations}(iv)).

By considering now $\dim{\rm Der}_{(1,0,1)}$ (see Table~\ref{TableDimmAnn}), we can exclude the possibility of degeneration from $\mathfrak a_{11}$ to each of the members of the set $\{\mathfrak a_7, \mathfrak a_{10}, \mathfrak c_3,\mathfrak l_1\}\cup\{\mathfrak a_3(\kappa)\colon \kappa\in\mathbb K\}$.
Moreover, the following holds:

{\bf Claim.} Algebra $\mathfrak a_{11}$ degenerates to each one of the members of the set $\{\mathfrak a_8, \mathfrak a_9, \mathfrak c_4, \mathfrak a_2, \mathfrak c_1, \mathfrak a_0\}$.

{\bf Proof of Claim.} From the transitivity of degenerations it suffices to prove that  $\mathfrak a_{11}\to \mathfrak c_4$, $\mathfrak a_{11}\to \mathfrak a_9$ and $\mathfrak a_{11} \to \mathfrak a_8$ (compare with Picture~\ref{Picture3DimDegenAssocPart2}).

We let $\boldsymbol\lambda=\Theta(\mathfrak a_{11})={\bf111}+{\bf322}+{\bf333}$, (see Table~\ref{Table3dimAssocClass}).
We also define, for $t\in\mathbb C-\{0\}$, matrices $g_1(t)$, $g_2(t)$ and $g_3(t)\in \GlC$ as follows:
\[
g_1(t)=\begin{pmatrix}
1&0&0\\
0&t&0\\
0&0&t
\end{pmatrix}, \quad
g_2(t)=\begin{pmatrix}
t&0&0\\
0&1&0\\
0&0&1
\end{pmatrix}
\quad\mbox{and}\quad
g_3(t)=\begin{pmatrix}
t&1&0\\
0&0&1\\
0&1&0
\end{pmatrix}.
\]
Then, (i) $\boldsymbol\lambda g_1(t)={\bf111}+t({\bf322})+t({\bf333})$,
(ii) $\boldsymbol\lambda g_2(t)=t({\bf111})+{\bf322}+{\bf333}$ and
(iii) $\boldsymbol\lambda g_3(t)=t({\bf111})+{\bf121}+{\bf211}+{\bf222}+{\bf233}$.
Invoking Lemma~\ref{LemmaA} (and an argument as in Example~\ref{ExamplApplA}) and comparing with the information given in Table~\ref{Table3dimAssocClass}, we get from the observations (i), (ii) and (iii) immediately above the existence, respectively, of degenerations   $\mathfrak a_{11}\to \mathfrak c_4$, $\mathfrak a_{11}\to \mathfrak a_9$ and $\mathfrak a_{11} \to \mathfrak a_8$.
This completes the degeneration picture inside the closed set $S_6$.

\bigskip
{\bf Step 7.} Finally we consider the set $S_7=\phi_1^{-1}(\{{\bf0}\}\cup\mathfrak l_1^o\cup\mathfrak m_4^o)\cap\phi_2^{-1}(\{J_2^o\cup J_8^o\cup \mathfrak c_3^o\cup \mathfrak c_2^o\cup \mathfrak c_1^o\cup \{{\bf0}\})\cap \mathcal A^s=\mathfrak a_1^o\cup\mathfrak a_2^o\cup\mathfrak a_9^o\cup\mathfrak a_{10}^o\cup\mathfrak c_3^o\cup\mathfrak c_2^o\cup\mathfrak c_1^o\cup \mathfrak l_1^o\cup\{{\bf0}\} \cup (\bigcup\limits_{\kappa\in\mathbb K}\mathfrak a_3(\kappa)^o)$. 

Let $\boldsymbol\nu=\Theta(\mathfrak a_1)={\bf111}+{\bf122}+\bf{212}+{\bf133}+{\bf313}+{\bf232}-{\bf322}+{\bf331}\in\boldsymbol\Lambda$, so $\boldsymbol\nu$ is the structure vector of algebra $\mathfrak a_1$ relative to the standard basis $(e_1,e_2,e_3)$ of the underlying $\mathbb C$-vector space $V$.
Consider now the basis $(f_1,f_2,f_3)$ of $V$ defined by $f_1=\frac12(e_1-e_3)$, $f_2=\frac12e_2, f_3=\frac12(e_1+e_3)$ and observe that the subspaces $\mathbb C$-span$(f_2)$, $\mathbb C$-span$(f_1, f_2)$ and $\mathbb C$-span$(f_2, f_3)$ are all ideals of~$\mathfrak a_1$.
Let $\boldsymbol\nu'$ be the structure vector of $\mathfrak a_1$ relative to the basis $(f_1,f_2,f_3)$.
Then $\boldsymbol\nu'={\bf111}+{\bf122}+{\bf232}+{\bf333}\,(\in O(\boldsymbol\nu))$.
Now define, for $t\in\mathbb C-\{0\}$, matrices $g_1(t)$, $g_2(t)\in\GlC$ by
\[
g_1(t)=\begin{pmatrix}
t&0&0\\
0&1&0\\
0&0&1
\end{pmatrix},
\quad g_2(t)=\begin{pmatrix}
1&0&0\\
0&1&0\\
0&0&t
\end{pmatrix}.
\]
It follows that $\boldsymbol\nu'g_1(t)=t({\bf111})+t({\bf122})+{\bf232}+{\bf333}\in O(\boldsymbol\nu)$ and that $\boldsymbol\nu'g_2(t)={\bf111}+{\bf122}+t({\bf232})+t({\bf333})\in O(\boldsymbol\nu)$, for $t\in\mathbb C-\{0\}$.
Invoking Lemma~\ref{LemmaA} (and an argument as in Example~\ref{ExamplApplA}) together with the information on the commutation relations given in~Table~\ref{Table3dimAssocClass}, we can see that the expression for $\boldsymbol\nu'g_1(t)$ (resp., $\boldsymbol\nu'g_2(t)$) obtained above, leads to a degeneration $\mathfrak a_1\to \mathfrak a_{10}$ (resp., $\mathfrak a_1\to\mathfrak a_9$).
[Observe that by setting $\mathfrak b=\mathbb C$-span$(f_2,f_3)$ (resp., $\mathfrak b=\mathbb C$-span$(f_1,f_2)$), the existence of the above degenerations can also be established using Lemma~\ref{LemmaC},
(compare also with Remark~\ref{Remark2.9}).]
From the transitivity of degenerations we also get $\mathfrak a_1\to\mathfrak a_2$,  (and $\mathfrak a_1\to\mathfrak c_1$, $\mathfrak a_1\to\mathfrak a_0$), see Picture~\ref{Picture3DimDegenAssocPart2}.

Next, we consider the structure vector $\boldsymbol\nu g_2(t)={\bf111}+{\bf122}+{\bf212}+{\bf133}+{\bf313}+t({\bf232})-t({\bf322})+t^2({\bf331})\in O(\boldsymbol\nu)$.
Applying Lemma~\ref{LemmaA} (compare also with Example~\ref{ExamplApplA}) we get $\mathfrak a_1\to\mathfrak c_2$.

Now let $S=\mathfrak c_3^o\cup\mathfrak l_1^o\cup(\bigcup\limits_{\kappa\in\mathbb K}\mathfrak a_3(\kappa)^o)$.
In order to complete the degeneration picture inside the set $S_7$ it remains to examine whether there is a degeneration from $\mathfrak a_1$ to any one of the members of the set $S$.
It is useful to observe that $S\subseteq \mathcal B\cap\mathcal A^s\subseteq (\mathcal T\cap\tilde{\mathcal T})\cap \mathcal A^s$ (see Remark~\ref{Remark5.1}(ii)).
In addition, $\dim{\rm ann}_L\mathfrak a=1=\dim{\rm ann}_R\mathfrak a$ for all $\mathfrak a\in \Theta^{-1}(S)$, see Table~\ref{TableDimmAnn}.

For this investigation it will be convenient to consider $\boldsymbol\lambda\in O(\boldsymbol\nu)$, where $\boldsymbol\lambda=\boldsymbol\nu'g$ with $g=${\footnotesize$\begin{pmatrix}1&0&0\\0&0&1\\0&1&0\end{pmatrix}$}$\in\GlC$, and then examine the structure vector $\boldsymbol\lambda b$ where $b=(b_{ij})\in\GlC$ is upper triangular (so $b_{ij}=0$ if $i>j$ and $b_{11}b_{22}b_{33}\ne0$).

Now $\boldsymbol\lambda={\bf111}+{\bf133}+{\bf323}+{\bf222}\in O(\boldsymbol\nu)$ and $\boldsymbol\lambda b=b_{11}({\bf111})+ b_{12}({\bf121})+\frac{b_{12}b_{23}}{b_{22}}({\bf131}) + \frac{-b_{11}b_{23}}{b_{22}}({\bf132})+b_{11}({\bf133})+b_{12}({\bf211})+\frac{b_{12}(b_{12}-b_{22})}{b_{11}}({\bf221})+b_{22}({\bf222})+ \frac{b_{12}b_{23}(b_{12}-b_{22})}{b_{11}b_{22}}({\bf231})+ \frac{-b_{23}(b_{12}-b_{22})}{b_{22}}({\bf232})+b_{12}({\bf233})+b_{13}({\bf311})+ \frac{b_{13}(b_{12}-b_{22})}{b_{11}}({\bf321})+b_{22}({\bf323})+\frac{b_{13}b_{23}(b_{12}-b_{22})}{b_{11}b_{22}}({\bf331})+\frac{-b_{13}b_{23}}{b_{22}}({\bf332}) +(b_{13}+b_{23})({\bf333})$.

As a consequence, the following is a list of polynomials in ${\bf I}(\boldsymbol\lambda B)$, where $B$ is the subgroup of all upper triangular matrices in $\GlC$:
\begin{gather*}
X_{112},\ X_{113},\ X_{122},\ X_{123},\ X_{212},\ X_{213},\ X_{223},\ X_{312},\ X_{313},\ X_{322},\\
X_{121}-X_{211},\ X_{111}-X_{133},\ X_{222}-X_{323},\ X_{121}-X_{233},\\
X_{131}X_{311}+X_{332}X_{211},\ X_{331}X_{221}-X_{231}X_{321},\\
X_{331}X_{211}-X_{231}X_{311},\ X_{311}X_{221}-X_{211}X_{321}.
\end{gather*}

Observe that $\overline{\boldsymbol\lambda B}\subseteq\mathcal A^s$ since $\boldsymbol\lambda B\subseteq\mathcal A^s$ and $\mathcal A^s$ is closed.
Hence, from $\mathcal B\cap \mathcal A^s\subseteq (\mathcal T\cap\tilde{\mathcal T})\cap \mathcal A^s$, we get $(\overline{\boldsymbol\lambda B})\cap\mathcal B\subseteq\mathcal T\cap\tilde{\mathcal T}$.
Now let $\boldsymbol\mu=(\mu_{ijk})\in(\overline{\boldsymbol\lambda B})\cap\mathcal B$.
Then $\boldsymbol\mu\in\mathcal T\cap\tilde{\mathcal T}$ so, for $1\le i\le 3$, we have $\sum_{j=1}^3\mu_{ijj}=0$ and $\sum_{j=1}^3\mu_{jij}=0$ (see Section~\ref{SectionPhi1Phi2} and Remark~\ref{Remark5.1}(ii)).
Moreover, $\ev_{\boldsymbol\mu}(f)=0$ for all polynomials $f\in{\bf I}(\boldsymbol\lambda B)$.
This forces $\mu_{112}=\mu_{113}=\mu_{122}=\mu_{123}=\mu_{212}=\mu_{213}=\mu_{223}=\mu_{312}=\mu_{313}=\mu_{322}=0$ and $\mu_{233}=\mu_{121}=\mu_{211}$, $\mu_{111}=\mu_{133}$, $\mu_{222}=\mu_{323}$.
Invoking the conditions forced by $\boldsymbol\mu=(\mu_{ijk})\in\mathcal T$, we also get $\mu_{111}=\mu_{133}=0$ (since $\mu_{122}=0$ and $\mu_{111}=\mu_{133}$), $2\mu_{211}+\mu_{222}=0$ (since $\mu_{211}=\mu_{233}$) and $\mu_{311}+\mu_{333}=0$ (since $\mu_{322}=0$).
So the defining conditions for $\boldsymbol\mu=(\mu_{ijk})$ to belong to $\tilde{\mathcal T}$ now give $2\mu_{222}+\mu_{211}=0$ (since $\mu_{222}=\mu_{323}$ and $\mu_{121}=\mu_{211}$) and $\mu_{131}+\mu_{232}+\mu_{333}=0$.
Also observe that the conditions $2\mu_{211}+\mu_{222}=0$, $2\mu_{222}+\mu_{211}=0$, $\mu_{222}=\mu_{323}$ and $\mu_{211}=\mu_{121}=\mu_{233}$ force $\mu_{211}=\mu_{222}=\mu_{121}=\mu_{233}=\mu_{323}=0$.

Next, we invoke the fact that $\boldsymbol\mu\in\mathcal B$, so for $1\le i,j,k,m\le n$ we have $\sum_{l=1}^n\mu_{ijl}\mu_{lkm}=0$ (see Section~\ref{SectionPhi1Phi2}).
From $\mu_{331}\mu_{133}+\mu_{332}\mu_{233}+\mu_{333}\mu_{333}=0$ we get $\mu_{333}=0$ (since $\mu_{133}=0=\mu_{233}$), and hence $\mu_{311}=0$ since $\mu_{311}+\mu_{333}=0$.

Moreover, from $\mu_{131}+\mu_{232}+\mu_{333}=0$ we now get $\mu_{131}+\mu_{232}=0$.
Also the fact that $X_{331}X_{221}-X_{231}X_{321}\in{\bf I}(\boldsymbol\lambda B)$ ensures that $\mu_{331}\mu_{221}-\mu_{231}\mu_{321}=0$.

Going back once more to the defining conditions for the algebraic set $\mathcal B$ and combining with various restrictions already obtained above we get the following further constraints on the coefficients $\mu_{ijk}$:
\begin{gather*}
\mu_{221}\mu_{131}=0,\ \mu_{221}\mu_{132}=0,\ \mu_{131}^2+\mu_{132}\mu_{231}=0,\ \mu_{321}\mu_{131}=0,\ \mu_{321}\mu_{132}=0,\\
\mu_{332}\mu_{221}=0,\ \mu_{331}\mu_{131}+\mu_{332}\mu_{231}=0,\ \mu_{331}\mu_{132}+\mu_{332}\mu_{232}=0.\qquad\qquad\qquad\qquad (*)
\end{gather*}
Summing up, the assumption $\boldsymbol\mu=(\mu_{ijk})\in(\overline{\boldsymbol\lambda B})\cap\mathcal B$ gives the constraints  that the only coefficients $\mu_{ijk}$ which can possibly be non-zero are
\[
\mu_{131},\ \mu_{132},\ \mu_{221},\ \mu_{231},\ \mu_{232},\ \mu_{321},\ \mu_{331}\ \mbox{and}\ \mu_{332}
\]
and these satisfy the conditions~$(*)$ together with the conditions $\mu_{131}+\mu_{232}=0$ and $\mu_{331}\mu_{221}-\mu_{231}\mu_{321}=0$.

We will consider the cases (i) $\mu_{131}\ne0$ and (ii) $\mu_{131}=0$ separately.

Case (i): $\mu_{131}\ne0$.
Then $\mu_{232}=-\mu_{131}\,(\ne0)$ and $\mu_{221}=\mu_{321}=0$.
Hence, all coefficients $\mu_{ijk}$ are zero except possibly $\mu_{131}$, $\mu_{132}$, $\mu_{231}$, $\mu_{232}$, $\mu_{331}$ and $\mu_{332}$.
This gives $\mathbb C$-span$(e_1,e_2)\subseteq {\rm ann}_R\Theta^{-1}(\boldsymbol\mu)$.
In particular, $\boldsymbol\mu\not\in S$.

Case (ii): $\mu_{131}=0$ (hence $\mu_{232}=0$ as well).

In this case, all coefficients $\mu_{ijk}$ are zero except possibly $\mu_{132}$, $\mu_{221}$, $\mu_{231}$, $\mu_{321}$, $\mu_{331}$ and $\mu_{332}$ and we have the constraints:
$\mu_{221}\mu_{132}=0$, $\mu_{132}\mu_{231}=0$, $\mu_{321}\mu_{132}=0$, $\mu_{332}\mu_{221}=0$, $\mu_{332}\mu_{231}=0$, $\mu_{331}\mu_{132}=0$ and $\mu_{331}\mu_{221}-\mu_{231}\mu_{321}=0$.

We will consider the subcases (ii)(a) $\mu_{132}\ne0$ and (ii)(b) $\mu_{132}=0$ separately:

Subcase (ii)(a): $\mu_{132}\ne0$:
Then, $\mu_{221}=\mu_{231}=\mu_{321}=\mu_{331}=0$.
Again we get that $\mathbb C$-span$(e_1,e_2)\subseteq {\rm ann}_R\Theta^{-1}(\boldsymbol\mu)$ so $\boldsymbol\mu\not\in S$ in this subcase.

Subcase (ii)(b): $\mu_{132}=0$.
In this subcase only $\mu_{221}$, $\mu_{231}$, $\mu_{321}$, $\mu_{331}$ and $\mu_{332}$ can possibly be non-zero and we have the constraints:
$\mu_{332}\mu_{221}=0$, $\mu_{332}\mu_{231}=0$ and $\mu_{331}\mu_{221}-\mu_{231}\mu_{321}=0$.
If $\mu_{332}\ne0$, then $\mu_{221}=0=\mu_{231}$.
Hence $\mathbb C$-span$(e_1,e_2)\subseteq {\rm ann}_L\Theta^{-1}(\boldsymbol\mu)$ leading to $\boldsymbol\mu\not\in S$.
So we assume that $\mu_{332}=0$ as well, leaving us with $\mu_{221}$, $\mu_{231}$, $\mu_{321}$, $\mu_{331}$ as the only coefficients which can possibly be non-zero and also satisfying the constraint $\mu_{331}\mu_{221}-\mu_{231}\mu_{321}=0$.

Now the matrix $\begin{pmatrix}\mu_{221}&\mu_{321}\\ \mu_{231}&\mu_{331}\end{pmatrix}\in M_2(\mathbb C)$ has determinant 0 (and hence has rank $<2$) so there exists a matrix $(\beta\ \gamma)\in\mathbb C^{1\times 2}-\{(0\ 0)\}$ satisfying $(\beta\ \gamma)\begin{pmatrix}\mu_{221}&\mu_{321}\\ \mu_{231}&\mu_{331}\end{pmatrix}=(0\ 0)$.
Inside $\Theta^{-1}(\boldsymbol\mu)$ we get: $[e_2,\beta e_2+\gamma e_3]=\beta\mu_{221}e_1+\gamma\mu_{231}e_1=0$ and $[e_3,\beta e_2+\gamma e_3] =\beta\mu_{321}e_1+\gamma\mu_{331}e_1=0$.
Hence, $\mathbb C$-span$(e_1,\beta e_2+\gamma e_3)\subseteq{\rm ann}_R\Theta^{-1}(\boldsymbol\mu)$ giving $\boldsymbol\mu\not\in S$ in this final subcase also.

We conclude that $(\overline{\boldsymbol\lambda B})\cap S\,(=(\overline{\boldsymbol\lambda B})\cap\mathcal B\cap S)=\varnothing$.
From Remark~\ref{RemarkBorelSubgroup}(ii) with $U=S$ we finally get that there is no degeneration from $\mathfrak a_1$ to any one of the members of the set $\{\mathfrak c_3,\mathfrak l_1\}\cup\{\mathfrak a_3(\kappa)\colon \kappa\in\mathbb K\}$ (information which cannot so easily be obtained by just considering various algebra invariants).
This completes the degeneration picture inside the closed set $S_7$.

\medskip
In view of Remark~\ref{Remark5.1}(iv), in order to obtain the complete degeneration picture inside the algebraic set $\mathcal A^s$ it is enough to combine the information given in Picture~\ref{Picture3DimJordChild} together with information obtained in Steps 1 to 7 of the above process (see Picture~\ref{Picture3DimDegenAssocFull}).

\begin{center}
\begin{tikzpicture}
\matrix [column sep=7mm, row sep=5mm] {
 & \node (c11)  {$\mathfrak c_{11}$}; & & & & & \\
 \node (c9)  {$\mathfrak c_{9}$}; & & \node (c10)  {$\mathfrak c_{10}$}; & & & & \\
 \node (c7)  {$\mathfrak c_{7}$}; &  \node (c6)  {$\mathfrak c_{6}$}; & \node (c8)  {$\mathfrak c_{8}$}; & \node (a11)  {$\mathfrak a_{11}$}; & & \node (a12)  {$\mathfrak a_{12}$}; & &\node (a1)  {$\mathfrak a_{1}$}; &
 \\
 & \node (c5)  {$\mathfrak c_{5}$};   & & \node (a9)  {$\mathfrak a_{9}$};& & &\node (a7)  {$\mathfrak a_{7}$}; & \node (a8)  {$\mathfrak a_{8}$};  &  \node (a10)  {$\mathfrak a_{10}$};
\\
 \node (a6)  {$\mathfrak a_{6}$}; & \node (a3k2)  {$\mathfrak a_{3}(\kappa=2)$}; &   \node (a3kNot2)  {$\mathfrak a_{3}(\kappa\ne2)$}; &\node (c3)  {$\mathfrak c_{3}$}; &\node (c2)  {$\mathfrak c_{2}$}; & \node (c4)  {$\mathfrak c_{4}$}; &\node (a2)  {$\mathfrak a_{2}$}; & & &
 \\
 \node (a4)  {$\mathfrak a_4$}; & &\node (l1)  {$\mathfrak l_{1}$};  & \node (c1)  {$\mathfrak c_{1}$}; &\node (a5)  {$\mathfrak a_5$}; & & & &
   \\
 & & \node (a0)  {$\mathfrak a_{0}$}; &  & & & & &
 \\
};
\draw[->, thin] (c11) -- (c9);
\draw[->, thin] (c11) -- (c10);
\draw[->, thin] (c10) -- (c7);
\draw[->, thin] (c10) -- (c8);
\draw[->, thin] (c9) -- (c6);
\draw[->, thin] (c9) -- (c7);
\draw[->, thin] (c9) -- (c8);
\draw[->,  thin] (c6) -- (c2);
\draw[->,  thin] (c6) -- (c5);
\draw[->, thin] (c7) -- (c5);
\draw[->, thin] (a1) -- (a9);
\draw[->, thin] (a1) -- (a10);
\draw[->, thin] (a1) -- (c2);
\draw[->,  thin] (c8) -- (c4);
\draw[->,  thin] (c8) -- (c5);
\draw[->, thin] (a11) -- (a9);
\draw[->, thin] (a11) -- (a8);
\draw[->, thin] (a11) -- (c4);
\draw[->, thin] (a12) -- (a7);
\draw[->, thin] (a12) -- (a10);
\draw[->, thin] (a12) -- (c4);
\draw[->, thin] (a10) -- (a2);
\draw[->, thin] (a8) -- (a2);
\draw[->, thin] (a7) -- (a2);
\draw[->,  thin] (c5) -- (c3);
\draw[->, thin] (a9) -- (a2);
\draw[->, thin] (a3kNot2) -- (c1);
\draw[->, thin] (a3k2) -- (c1);
\draw[->, thin] (a3k2) -- (l1);
\draw[->, thin] (a6) -- (l1); 
\draw[->, thin] (c2) -- (c1);
\draw[->, thin] (c3) -- (c1);
\draw[->, thin] (c4) -- (c1);
\draw[->, thin] (a2) -- (c1);
\draw[->, thin] (l1) -- (a0);
\draw[->, thin] (a4) -- (a0);
\draw[->, thin] (c1) -- (a0);
\draw[->, thin] (a5) -- (a0);
\end{tikzpicture}

{\footnotesize
\refstepcounter{pict}\label{Picture3DimDegenAssocFull}
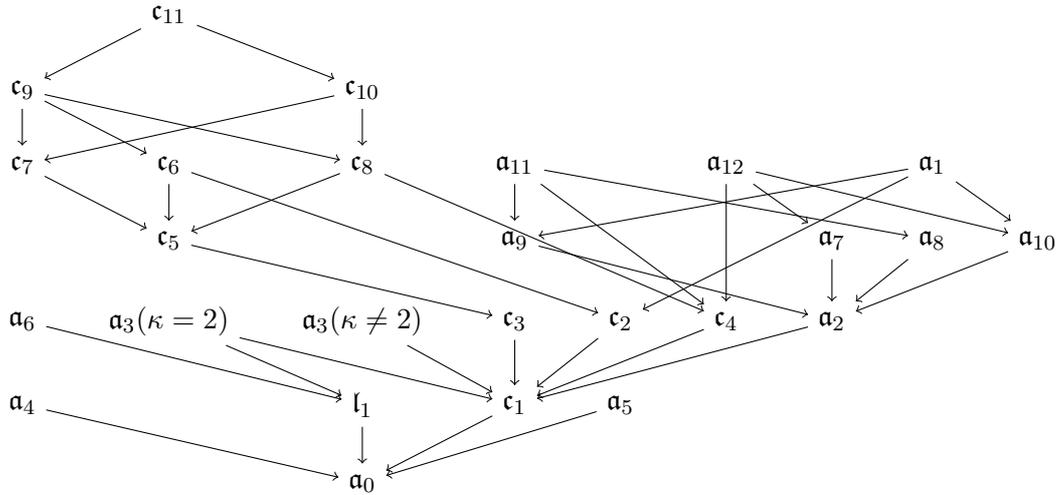
Picture~\ref{Picture3DimDegenAssocFull}. Degenerations of 3-dimensional complex associative algebras (with $\kappa\in\mathbb K$).
}
\end{center}

\end{document}